\documentclass[12pt]{article}
\RequirePackage{amsthm,amsmath}
\RequirePackage[numbers]{natbib}
\RequirePackage[colorlinks,citecolor=blue,urlcolor=blue]{hyperref}

\usepackage{amsfonts}
\usepackage{url}
\usepackage{graphicx}
\usepackage{fullpage}
\usepackage{bbm}

\def\T={\buildrel {\scriptscriptstyle\triangle} \over =}

\def\Tr{{\mathrm{Tr}}}
\def\Ex{{\mathbb{E}}}
\def\Prob{\mathbb{P}}
\def\p{{^{\prime}}}
\def\mc{\mathcal}

\def\ba{\begin{array}}
\def\ea{\end{array}}
\def\l{\left(}
\def\r{\right)}
\def\lb{\left\{}
\def\rb{\right\} }

\def \orb {\mc{O}}

\def \noff{\|_{1, \text{off}}}

\def\one{\mathbbm{1}}
\def\iid{i.i.d.}
\def\Sym{\text{Sym}}
\def\kron{\otimes}
\def\bfm{\begin{fmpage}{1.01\linewidth}}
\def\efm{\end{fmpage}}
\def\mf{\mathfrak}

\newtheorem{theorem}{Theorem}[section]
\newtheorem{lemma}{Lemma}[section]
\newtheorem{corollary}{Corollary}[section]

\newtheorem{definition}{Definition}[section]

\newtheorem{assumption}{Assumption}[section]
\newtheorem{remark}{Remark}[section]
\newtheorem{example}{Example}[section]

\newcommand{\C}{\mathbb{C}}
\newcommand{\R}{\mathbb{R}}
\newcommand{\Z}{\mathbb{Z}}

\newsavebox{\fmbox}
\newenvironment{fmpage}[1]
{\begin{lrbox}{\fmbox}\begin{minipage}{#1}}
{\end{minipage}\end{lrbox}\fbox{\usebox{\fmbox}}}

\title{Group Symmetry and Covariance Regularization}
\author{Parikshit Shah and Venkat Chandrasekaran\\
\small email: pshah@discovery.wisc.edu; venkatc@berkeley.edu \normalsize}

\date{}

\begin{document}
\maketitle

\begin{abstract}
Statistical models that possess symmetry arise in diverse settings such as random fields associated to geophysical phenomena, exchangeable processes in Bayesian statistics, and cyclostationary processes in engineering. We formalize the notion of a symmetric model via group invariance. We propose projection onto a group fixed point subspace as a fundamental way  of regularizing covariance matrices in the high-dimensional regime. In terms of parameters associated to the group we derive precise rates of convergence of the regularized covariance matrix and demonstrate that significant statistical gains may be expected in terms of the sample complexity. We further explore the consequences of symmetry on related model-selection problems such as the learning of sparse covariance and inverse covariance matrices. We also verify our results with simulations.
\end{abstract}
%

\section{Introduction} \label{sec:introduction}
An important feature of many modern data analysis problems is the small number of samples available relative to the dimension of the data.  Such \emph{high-dimensional} settings arise in a range of applications in bioinformatics, climate studies, and economics.  A fundamental problem that arises in the high-dimensional regime is the poor conditioning of sample statistics such as sample covariance matrices \cite{BickelLevina1} ,\cite{BickelLevina2}. Accordingly, a fruitful and active research agenda over the last few years has been the development of methods for high-dimensional statistical inference and modeling that take into account \emph{structure} in the underlying model.  Some examples of
structural assumptions on statistical models include models with a few latent factors (leading to low-rank covariance matrices) \cite{Fan_Fan_Lv_2007}, models specified by banded or sparse covariance matrices \cite{BickelLevina1}, \cite{BickelLevina2}, and Markov or graphical models \cite{Lauritzen_book,RWRY,Meinshausen}.

The focus of this paper is on exploiting \emph{group symmetries} in covariance matrices.  Models in which the distribution of a collection of random variables is invariant under certain permutations of the variables have been explored in many diverse areas of statistical research \cite{Lauritzen_symmetry}.  Symmetry assumptions have played a prominent role within the context of covariance matrices and multivariate Gaussians \cite{InvariantNormalModels} and these assumptions are of interest in numerous applications (see \cite{Lauritzen_symmetry} for a detailed list).



We systematically investigate the statistical and computational benefits of exploiting symmetry in high-dimensional covariance estimation.  As a very simple example consider the estimation of the variances of $p$ independent random variables from $n$ samples of the variables.  Since the variables are independent, this problem is equivalent to estimating the variances of each variable separately from $n$ samples.  Suppose we are additionally given that the variables are identically distributed -- we now need to estimate just a single parameter from what are effectively $n \times p$ samples.  This very elementary example demonstrates the potential for improvements in the sample complexity as well as computational complexity in models with symmetries.  We generalize this basic insight to much more complicated settings in which the distribution underlying a collection of random variables may be symmetric with respect to an arbitrary subgroup of the symmetric group.  More formally, we investigate multivariate Gaussians specified by \emph{invariant covariance matrices}:
\begin{equation*}
\Sigma = \Pi_g \Sigma \Pi_g^T ~~~~~ \forall \Pi_g \in \mathfrak{G},
\end{equation*}
where $\mathfrak{G}$ is some subgroup of the symmetric group, i.e., the group of all permutations.  Associated to each subgroup $\mathfrak{G}$ of the symmetric group is the \emph{fixed-point subspace} $\mathcal{W}_\mathfrak{G}$ of matrices that is invariant with respect to conjugation by each element of $\mathfrak{G}$.  This subspace plays a fundamental role via Schur's lemma from representation theory in our analysis of the benefits in sample complexity and computational complexity of regularizing invariant covariance matrices.

The general advantages of symmetry exploitation are numerous
\begin{itemize}
\item Problem size: One advantage is that when symmetry is incorporated in the model, the problem size often reduces, and the new instance can have significantly fewer number of variables and constraints, which can lead to dramatic computational speed-ups. This is exploited for example in finite-element methods for partial differential equations \cite{fassler}. 
\item Better statistical properties: As we will see in this paper, exploiting symmetry also leads to statistical gains in terms of order-of-magnitude gains in the sample complexity of model selection. 
\item Numerical benefits: Another advantage is the removal of \emph{degeneracies} in the problem that arises from the high multiplicity of eigenvalues that is typically associated to symmetric models. Such multiplicities are a common source of difficulties in numerical methods, and they can be properly handled by suitably exploiting symmetry in the problem. Symmetry-aware methods in general have better numerical conditioning, are more reliable and lead to faster and more accurate solutions.
\end{itemize}


This paper consists of three main technical contributions.  First, we precisely quantify in terms of properties of $\mathfrak{G}$ the improvement in rates of convergence when the sample covariance matrix (obtained from $n$ samples) is projected onto $\mathcal{W}_\mathfrak{G}$.  Our analysis holds for the spectral or operator norm, the Frobenius norm, and the $\ell_\infty$ norm (maximum absolute entry).  Second, we study the implications of these improved rates of convergence by specializing recent results on covariance estimation of Bickel and Levina \cite{BickelLevina2}, and of Ravikumar et al. \cite{RWRY} to our setting with invariant covariance matrices.  These results quantitatively demonstrate the benefits of taking symmetry structure into account in covariance estimation tasks.  Finally, we discuss the computational benefits of incorporating symmetry in covariance regularization.  Specifically, we describe how large-dimensional convex programs based on regularized maximum-likelihood can be simplified significantly to smaller-sized convex programs by incorporating prior information about symmetry.

Symmetry and its consequences in the context of statistics, and more broadly in computational mathematics, has been well studied. We mention some related work here. In the Gaussian setting, work has been done related to testing of symmetry in statistical models \cite{wilks,olkin_press,votaw}. In the context of Markov random fields symmetry restrictions have been studied in \cite{Whittle,Besag1,Besag2,Hylleberg,Andersson2,Madsen}. We also mention the work of H${\o}$jsgaard and Lauritzen \cite{Lauritzen_symmetry}, who study edge and vertex symmetries for covariance and inverse covariance (graphical models) and exploitation of symmetry for inference problems. They also cite interesting examples from the social sciences (such as distribution of test scores among students and results of psychological tests) where symmetric structure has been experimentally observed. More broadly, symmetry has played an important role in other computational problems such as optimization \cite{Vallentin1,gatermann_parrilo,deKlerk,Kanno}, in the computation of kissing numbers and sphere-packing problems \cite{Vallentin2}, in coding theory \cite{Schrijver}, in truss topology optimization \cite{Bai}, and in the design of rapidly mixing Markov chains \cite{Boyd}.
%
%

In the remainder of Section~\ref{sec:introduction} we explain a few specific applications where exploiting symmetry in statistical models is of interest. We also set up some of the notation that will be used throughout the paper. In Section~\ref{sec:symmetric_models} we introduce the necessary group theoretic preliminaries, the notion of $\mf{G}$-invariant covariance matrices, and related statistical quantities. In Section~\ref{sec:3} we state some of our main results that quantify the statistical gains when symmetry is exploited judiciously. In Section~\ref{sec:sparse_model_selection}  we extend our results from Section~\ref{sec:3} to selecting sparse covariance and inverse covariance models. In Section~\ref{sec:computational_aspects}, we make some brief remarks regarding how computational benefits can be gained when selecting symmetric models. In Section~\ref{sec:experimental_results}, we present some simulation results that clarify the statistical gains one may expect when symmetry is exploited. In Section~\ref{sec:conclusion}, we state some concluding remarks.

\subsection{Applications}
In this section we motivate the problem under consideration by briefly describing a few applications in which estimation of symmetric covariance structure arises naturally.
\subsubsection{Random Fields in Physical Phenomena}
Many signals that capture the spatial variability of natural phenomena are described in the framework of random fields \cite{Moura}. Random fields arise in a variety of engineering contexts, for example the temperature distribution of materials, the circulation of ocean fields in oceanography, or the transport of groundwater in hydrology \cite{Moura}. A classical model is associated with the well-known Poisson's equation:
\begin{equation} \label{eq:poisson}
\nabla^2 \phi(x) =f(x) \qquad x \in \R^d
\end{equation}
where $\nabla^2$ is the Laplace operator. In hydrology, this equation governs the steady state flow of a fluid in $\R^3$, whereas in electrostatics it governs the distribution of the potential in a variety of settings such as transmission lines \cite{Moura}. Under the right technical conditions, if $f(x)$ is a white noise process, the stochastic process $\phi(x)$ in \eqref{eq:poisson} is a Gauss-Markov random field \cite{GMRFbook}. While Poisson's equation \eqref{eq:poisson} is a \emph{sample-path} representation of the process $\phi(x)$, a particularly useful representation of this process is the covariance $R(x_1,x_2)$ associated to this random field. Indeed the covariance $R(x_1,x_2)$ is the Green's function of the biharmonic equation, and its properties have been studied in detail by Moura and Goswami \cite{Moura}.

A useful observation that is pertinent in our context is that the Laplacian as well as the biharmonic equation are invariant under the action of the orthogonal group, i.e. they are isotropic. If in addition the forcing process $f(x)$ has isotropic stochastic properties (one such example being a white noise process) and the boundary conditions of the physical process are also isotropic, then the covariance $R(x_1,x_2)$ will also be isotropic.

When such random fields are encountered in practice, the usual engineering approach is to discretize the domain so as to make the problem computationally tractable. The restriction of the covariance $R(x_1,x_2)$ to the finite discretization of the domain results is a covariance matrix $\hat{R}$. Recognizing the symmetry in the problem, it is often useful to construct a \emph{symmetry-preserving discretization} of the space \cite{symmetry_discretizations}. Such a modeling paradigm naturally leads to a covariance matrix that is invariant with respect to a discrete group $\mathfrak{G}$ which is a finite subgroup of the orthogonal group. A motivating question in this paper then is: \\
``Given samples of the stochastic process $\phi(x)$ on this discretized domain, can one exploit the symmetry of the space to obtain a good statistical approximation of the covariance matrix $\hat{R}$?''

\subsubsection{Bayesian Models}
A fundamental notion in Bayesian statistics is that of exchangeability. Indeed, in such settings one often deals with a collection of random variables $X_1, \ldots, X_n$ where it is not appropriate to assume that this collection is $\iid$. It is far more natural to assume instead that the collection of random variables is \emph{exchangeable} or \emph{partially exchangeable}. We give two such examples:
\begin{enumerate}
\item \textbf{Urn models:} Consider an urn with $n_r$ red balls and $n_b$ blue balls. A quantity $T$ balls are randomly drawn \emph{without replacement} from the urn. Let us define the random variable:
$$
X_i=\left\{ \begin{array}{cl} 0 &\text{ if the } i^{th} \text{ ball is red} \\
1 &\text{ if the } i^{th} \text{ ball is blue.} \ea \right.
$$
Clearly the random variables $X_1, \ldots, X_T$ are \emph{not} $i.i.d.$ On the other hand, it is well-known that this sequence of random variables is exchangeable, i.e. for any permutation $\sigma$ on the indices $1, \ldots, T$ and for any set $A$
$$
\Prob \l (X_1, \ldots, X_T) \in A \r=\Prob \l (X_{\sigma(1)}, \ldots, X_{\sigma(T)}) \in A \r.
$$
In the terminology developed in this paper, this distribution is invariant with respect to the group $\mathfrak{G}$ corresponding to the \emph{symmetric group}. Such urn models are commonly used metaphors for problems related to computing statistical properties of a population of objects.

Suppose the number of red and blue balls in the urn is unknown, and one wants to estimate the probability of drawing two consecutive blue balls in the first two draws empirically from multiple trials of $T$ draws each. It is an easy exercise to check that:
$$
\Ex \left[ X_1 X_2 \right]=\Prob \l X_1=1, X_2=1\r.
$$
By exchangeability $\Ex \left[ X_1 X_2\right]=\Ex \left[ X_i X_j\right]$ for any $i, j \in \left\{1, \ldots, T \right\}$. Hence the computation of the required probability reduces to estimating the empirical covariance matrix. It is important to exploit the underlying symmetry in this computation.

\item \textbf{Clinical tests:}
Suppose a clinical test is conducted on a set of $T$ patients wherein a drug is administered to the patients and a physiological response $x_1, x_2, \ldots, x_T$ is measured for each patient. The set of patients are divided into groups $T_1, \ldots, T_h$ of patients of similar physiological properties (say based on characteristics such as combinations of similar age, ethnicity, sex, and body size). It is often reasonable to assume that the distributions of the responses $X_i$ of patients within the same class are exchangeable (but not $\iid$). Sources of non-independence include latent factors such as residual effects of a previous experiment, physical effects such as the temperature, etc. Indeed from a Bayesian perspective it is more natural to assume that the responses are exchangeable as opposed to $\iid$. Thus, the distribution of $\l X_1, \ldots, X_n \r$ may be assumed to be invariant with respect to permutation of patient labels within the same class $T_i$. Specifically, if $\Sym_k$ denotes the symmetric group of permutations on $k$ indices then the the distribution of $\l X_1, \ldots, X_n \r$ is invariant under the action of the group $\mathfrak{G}=\Sym_{| T_1|} \times \ldots \times \Sym_{| T_h|}$. In such clinical tests, it is of importance to estimate the covariance matrix of the responses from a limited number of trials.
\end{enumerate}

\subsubsection{Cyclostationary processes}
An interesting special case of random phenomena with group symmetries are the class of well-studied stochastic processes known as \emph{cyclostationary processes} \cite{Gardner}. These processes arise in random phenomena that obey periodic stochastic properties, for example the yearly variation of temperature in a particular location. Such processes are frequently encountered in engineering and arise in a variety of contexts such as communication systems (presence of a sinusoidal component), radar (rotating elements), and the study of vibrations in mechanical engineering. Moreover, a special case of such processes are the classical stationary processes, which are ubiquitous in engineering. 

Formally, a discrete time stochastic process $\left\{ X_t \right\}_{t \in \Z}$ is said to be cyclostationary with period $N$ if for any measurable set $A$
$$
\Prob \l (X_{t_1}, \ldots X_{t_k}) \in A \r= \Prob \l X_{t_1+N}, \ldots X_{t_k+N}) \in A \r \; \text{ for all } t_1, \ldots, t_k \in \Z.
$$
Thus the joint probability distribution of such a process is invariant under shifts of $N$ time units. In the particular instance when $N=1$ we obtain a \emph{stationary} process.
An important statistical property associated to this process is the \emph{auto-correlation}
$$
R(t, \tau):=\Ex \left[ X_t X_\tau \right].
$$
If we restrict attention to the first $kN$ time units for some integers $k$, it is clear that the restriction of the autocorrelation admits a finite representation as a matrix $R \in \R^{kN \times kN}$. In the stationary case (i.e. $N=1$), from shift invariance it is also apparent that $R(t,\tau)$ depends \emph{only} on $|t-\tau|$, so that the resulting covariance matrix is \emph{Toeplitz}. A commonplace and particularly useful approximation for large Toeplitz matrices is to convert them to \emph{circulant} matrices \cite{circulant}. In the more general cyclostationary case (the period $N>1$), the matrix $R$ is block Toeplitz and approximated via a \emph{block circulant} one where each block is of size $N \times N$.

In group-theoretic terms, the covariance matrix $R$ is invariant under the action of the cyclic group of order $k$ (the group captures cyclic translations of the indices). This group is abelian, and its action on the index set is precisely the cyclic shift operator by $N$ units. As is well-known from the representation theory of the cyclic group, the \emph{Fourier transform} (block) diagonalizes circulant matrices, and the resulting (block) diagonal matrix is called the spectrum of the autocorrelation. These observations are classical results and we mention them here to point out an interesting special case of our framework.

\subsection{Notation}
We describe some of the notation that will be adopted through the rest of this paper. We will commonly deal with a random vector $X \in \mathbb{R}^p$, where $p$ is the dimension of the space. Occasionally, we will encounter the complex vector space $\C^p$. We will be interested in inferring statistical properties of this random vector from $\iid$ samples $X_1, \ldots, X_n$ where $n$ is the number of samples. We will study the problem in a \emph{high-dimensional} setting, i.e. where $p \gg n$ so that the dimension of the space is much larger than the number of samples. 

We will mostly consider \emph{normally distributed} random vectors $X$ with zero mean and covariance $\Sigma$, we will denote this by $\mc{N}(0, \Sigma)$. The inverse of the covariance is called the \emph{concentration matrix} and we use the notation $\Theta:=\Sigma^{-1}$. The probability of an event $A$ will be denoted by $\Prob(A)$ and the expectation of a random variable $Y$ is denoted by $\Ex[Y]$. The indicator function of an event $A$ is denoted by $\one(A)$. If $h(X)$ is a real valued function of a $\R^p$-valued random vector, we will use the notation that $h(X)=O_P(f(p,n))$ to mean that
$h(X) \leq C f(p,n)$ with probability exceeding $1-\frac{1}{p^\tau}$ for some constants $C$ and $\tau>0$.

Given a finite group $\mathfrak{G}$ which is a subgroup of the symmetric group, we consider the \emph{group action} of $\mathfrak{G}$ on the index set $\left\{1, \ldots, p \right\}$. Specifically, an element $g \in \mathfrak{G}$ can be interpreted as a permutation, and $g$ acts on $\left\{1, \ldots, p \right\}$ by permuting these indices. To every $g \in \mathfrak{G}$ one can associate a permutation matrix $\Pi_g \in \R^p$.

We will denote the positive-definiteness (respectively positive semidefiniteness) of a symmetric matrix $M$ by $M \succ 0$ ($M\succeq 0$). The space of symmetric positive definite matrices in $\R^{p \times p}$ will be denoted by $\mc{S}^p_{++}$. The maximum eigenvalue of a symmetric matrix $M$ is denoted by $\lambda_{\max}(M)$. Its \emph{trace} (sum of the diagonal entries) will be denoted by $\Tr(M)$. We will encounter different norms on the space of symmetric matrices:
\begin{enumerate}
\item Frobenius norm: Denoted by $\| M \|_F=\l \Tr(M^TM) \r^{\frac{1}{2}}$.
\item Spectral norm: Denoted by $\| M\|$ is the maximum singular value of $M$.
\item Elementwise $\ell_{\infty}$ norm: Denoted by $\| M \|_{\infty}= \max_{i,j} |M_{ij}| $.
\item The operator $\infty$ norm: Denoted by $\| M \|_{\infty, \infty}= \max_{i} \sum_{j}|M_{ij}|$.
\item Off-diagonal $\ell_1$ norm: Denoted by $\| M \|_{1,\text{off}}=\sum_{i,j, i \neq j}|M_{ij}|$.
\end{enumerate}
%
%

\section{Symmetric Models} \label{sec:symmetric_models}
\subsection{Symmetry, Groups, and Group Representations}
The central concepts that formalize the notion of symmetry mathematically are \emph{groups} and \emph{invariance}. Abstractly, a physical object or process (in our case a covariance matrix) is said to possess symmetry if it is invariant under the action of a group. For background on groups and group actions we refer the reader to any standard algebra textbooks, for example \cite{dummit,artin,serre}.  
In this paper we will be dealing with finite groups $\mathfrak{G}$, i.e. groups whose cardinality $|\mathfrak{G}|$ is finite. A basic result in group theory, Cayley's theorem, establishes that every finite group is isomorphic to a subgroup of the symmetric group (i.e. the group of permutations with the group operation being permutation composition).
Consequently, elements of finite groups $\mathfrak{G}$ can be conveniently represented via permutation matrices, and in our context these naturally induce an action of $\mathfrak{G}$ on matrices $M \in \R^{p \times p}$ via conjugation. 
 
Specifically, letting $\Pi_g$ denote the permutation matrix corresponding to $g \in \mathfrak{G}$, a \emph{group action} is a map $\varphi$: 
\begin{align*}
\varphi :\mf{G} \times R^{p \times p} &\rightarrow \R^{p \times p} \\
(g,M) &\mapsto \Pi_g M \Pi_g^T.
\end{align*}
 A matrix $M$ is said to be \emph{invariant} with respect to the group if $M=\Pi_g M \Pi_g^T$ for all $g \in \mf{G}$.
The set $$\mc{W}_\mathfrak{G}=\left\{X \in \mathbb{R}^{p \times p} \; | \; \Pi_g^T X \Pi_g =X \; \; \forall \Pi_g \in \mathfrak{G} \right\}$$is called the \emph{fixed point subspace} of the group $\mathfrak{G}$ and it is the set of all matrices that are invariant with respect to the group. Informally, this is the set of matrices that possess ``symmetry'' with respect to the group. 
We let $\mc{P}_{\mf{G}}(\cdot): \R^{p \times p} \rightarrow \mc{W}_{\mf{G}}$ denote the Euclidean projection operator that projects a matrix onto the fixed point subspace.

A fundamental tool in the study of invariance is the theory of group representations. We provide a brief review of concepts from group representation theory which are relevant for our discussion. Let $GL(\C^N)$ denote the general linear group of invertible matrices in $\C^{N \times N}$. A \emph{group representation} is a group homomorphism 
\begin{align*}
\rho:&\mathfrak{G} \rightarrow GL(\C^N) \\
& g \mapsto \rho(g).
\end{align*}  
The representation naturally induces an action of $\mathfrak{G}$ onto the vector space $\C^N$ via
\begin{align*}
\rho: &\mathfrak{G} \times \C^N \rightarrow \C^N \\
&(g, v) \mapsto \rho(g) v.
\end{align*}
Thus, the representation naturally associates to each $g \in \mathfrak{G}$ a linear transformation $\rho(g)$ that acts on $\C^N$. Since a linear transformation is characterized by its invariant subspaces, it is sufficient to study the invariant subspaces of the linear transformations $\rho(g)$ for $g\in \mathfrak{G}$. Indeed, if $\rho$ is a given representation and $W \subseteq \C^N$ is a subspace, we say that $W$ is \emph{$\mathfrak{G}$-invariant} with respect to $\rho$ if $\rho(g)W \subseteq W$ for every $g \in \mathfrak{G}$. We say that a representation $\rho$ is \emph{irreducible} if the only $\mathfrak{G}$-invariant subspaces are the trivial subspaces $0$ and $\C^N$. \emph{Schur's lemma}, a classical result in group representation theory, establishes that finite groups have only finitely many irreducible representations, and provides the following orthogonal decomposition theorem of $\C^N$ in terms of the $\mathfrak{G}$-invariant subspaces.
\begin{lemma}[Schur's lemma \cite{serre}] \label{lemma:schur}
For a finite group $\mathfrak{G}$ there are only finitely many inequivalent irreducible representations (indexed by $\mc{I}$) $\vartheta_1, \ldots, \vartheta_{|\mc{I}|}$ of dimensions $s_1, \ldots, s_{|\mc{I}|}$. The dimensions $s_i$ divide the group order $|\mathfrak{G}|$, and satisfy $\sum_{i \in \mc{I}} s_i^2=|\mf{G}|$. Every linear representation of $\mf{G}$ has a canonical decomposition
$$
\rho=m_1 \vartheta_1 \oplus \ldots \oplus m_{|\mc{I}|} \vartheta_{|\mc{I}|},
$$
where the $m_i$ are the multiplicities. Correspondingly, there is an \emph{isotypic} decomposition of $\C^N$ into invariant subspaces $W_i$:
$$
\C^N=W_1\oplus \ldots \oplus W_{|\mc{I}|}
$$
where each $W_i$ is again a direct sum of isomorphic copies $W_i=W_{i1} \oplus \ldots \oplus W_{i,m_i}$ with multiplicity $m_i$.
\begin{footnote}{
Schur's lemma, as stated classically, provides a decomposition over the complex field $\C^N$. However, a real version can be adapted in a straightforward manner \cite[pp. 106-109]{serre},\cite{gatermann_parrilo} and the irreducible real representations are called \emph{absolutely irreducible}.} 
\end{footnote}
\end{lemma}
A basis of this decomposition (that depends only on the group) transforming with respect to the matrices $\vartheta(g)$ is called symmetry adapted, and can be explicitly computed algorithmically \cite{serre,fassler}. This basis defines a change of coordinates by a unitary matrix $T \in \C^{N \times N}$. One of the main consequences of Schur's lemma is that the symmetry adapted basis \emph{block diagonalizes the fixed point subspace}, i.e. every matrix in the fixed point subspace is commonly diagonalized by $T$.
If $M \in \mc{W}_{\mf{G}}$ is a matrix in the fixed point subspace, then changing coordinates with respect to $T$ decompses $M$ into a block diagonal form as follows:
\begin{align} \label{eq:decomp}
T^{*}MT= \left[ \ba{ccc}
M_1 & & 0 \\
& \ddots & \\
0 & & M_{|\mc{I}|}
\ea \right] \qquad
M_i= \left[ \ba{ccc}
B_i & & 0 \\
& \ddots & \\
0 & & B_i
\ea \right].
\end{align}
In the above decomposition the diagonal blocks $M_i \in \R^{m_i s_i \times m_i s_i}$ can be further decomposed into $m_i$ repeated diagonal blocks $B_i \in \R^{s_i \times s_i}$ (recall that $s_i$ are the sizes of the irreducible representations and $m_i$ are the multiplicities). Thus, the symmetry restriction $M \in \mc{W}_{\mf{G}}$ reduces the degrees of freedom in the problem of interest. This observation plays a central role in our paper.

\begin{example} \label{example:1}
To fix ideas, we will consider the following three examples to illustrate our results throughout the paper.
\begin{enumerate}
\item \textbf{Cyclic group:} The cyclic group of order $p$ is isomorhpic to $\Z/p \Z$ (the group operation is the addition operation $+$). Let $M \in \R^{p \times p}$ be a circulant matrix of the form:
\begin{equation} \label{eq:circ}
M=\left[ \ba{cccc} 
M_1 & M_2 & \ldots & M_p \\
M_p & M_1 & \ldots & M_{p-1} \\
\vdots & \vdots & \ddots & \vdots \\
M_2 & M_3 & \ldots &M_1
\ea \right].
\end{equation}
If $\rho(k)$ is a permutation matrix corresponding to cyclic shifts of size $k$ of the indices $\left\{1, \ldots, p \right\}$, then it is clear that $\rho(k)M\rho(k)^T=M$. Indeed the fixed point subspace is precisely the set of circulant matrices and since the group is abelian, can be described as
$$
\mc{W}_{\mf{G}}=\left\{M \in \R^{p \times p} \; | \; M \rho(1) =\rho(1) M \right\}.
$$
(Since $1$ is the generator of the cyclic group, invariance with respect to $1$ implies invariance with respect to the other elements).

It is well known that if $T=\mc{F} \in \C^{p \times p}$ is the Fourier transform matrix, then $T^*MT$ is block diagonal. (The columns of $\mc{F}$ form the symmetry adapted basis). Here $|\mc{I}|=p$ (there are $p$ irreducible representations), $m_i=1$ and $s_i=1$ for all $i \in \mc{I}$. In our setting, covariance matrices of such structure arise from cyclostationary processes.
\item \textbf{Symmetric group:} 
Let us now consider $\mf{G}=\Sym_{p}$, the full symmetric group of permutations on $p$ indices. Let $M \in \R^{p \times p}$ be have the following structure:
\begin{equation} \label{eq:full_sym}
M=\left[ \ba{cccc} 
a & b & \ldots & b \\
b & a & \ldots & b \\
\vdots & \vdots & \ddots & \vdots \\
b & b & \ldots & a \ea \right].
\end{equation}
It is straightforward to check that $\Pi_g M=M \Pi_g$ for every permutation matrix $\Pi_g$, matrices satisfying this property constitute the fixed point subspace. It is well known that via the \emph{permutation representation} there is an orthogonal matrix 
\begin{footnote}{
In fact any orthogonal matrix $T$ whose first column is a multiple of the all ones vector, with the remaining columns span the orthogonal subspace will achieve the desired diagonalization.} 
\end{footnote} $T$ such that:
$$
T^*MT=
M=\left[ \ba{c|ccc} 
c_1 & 0 & \ldots & 0 \\ \hline
0 & c_2 & \ldots & 0 \\
\vdots & \vdots & \ddots & \vdots \\
0 & 0 & \ldots & c_2
\ea \right].
$$
Here $|\mc{I}|=2, s_1=1, m_1=1, s_2=1, m_2=p-1$. In our setting, covariance matrices with this structure arise when we have a completely exchangeable sequence of random variables.
\item \textbf{Product of symmetric groups:}
Let $T_1, \ldots, T_k$ be finite index sets where $|T_i|=r_i$. We consider the product of symmetric groups $\Sym_{|T_1|} \times \ldots \times \Sym_{|T_k|}$ acting on the index set $T_1 \times \ldots \times T_k$. Consider a block $k \times k$ matrix of the form:
\begin{equation} \label{eq:prod_sym}
M=\left[ \ba{cccc}
M_{11} & M_{12} & \ldots & M_{1k} \\
M_{12}^T & M_{22} & \ldots & M_{2k} \\
\vdots & \vdots & \ddots & \vdots \\
M_{1k}^{T} & M_{2k}^{T} & \ldots & M_{kk}
\ea \right],
\end{equation}
where for each $j \in \left\{1, \ldots, k \right\}$, $M_{jj} \in \R^{r_j \times r_j}$ is of the form \eqref{eq:full_sym}, and each $M_{ij}$ for $i \neq j$ is a constant multiple of the all ones matrix. Let $\sigma =(\sigma_1, \ldots, \sigma_k) \in \Sym_{|T_1|} \times \ldots \times \Sym_{|T_k|}$ be a permutation with this product structure. Let $\Pi_{\sigma}$ be its corresponding permutation matrix. Then it is clear that for a matrix of the form described above, $\Pi_\sigma M=M \Pi_\sigma$, and the set of matrices satisfying this relation constitute the corresponding fixed point subspace $\mc{W}_{\mf{G}}$. Using standard representation theory results, it is possible to show that there is an orthogonal matrix $T$ such that 
$$
T^*MT=
\left[ \ba{cccc} 
C & 0 & \ldots & 0 \\ 
0 & c_1 I_{r_1-1} & \ldots & 0 \\
\vdots & \vdots & \ddots & \vdots \\
0 & 0 & \ldots & c_k I_{r_k-1}
\ea \right].
$$
where $C \in \R^{k \times k}$ and $I_r \in \R^{r \times r}$ is the identity matrix of size $r$. (In fact, upto a permutation of the columns, $T$ can be chosen to be block diagonal where each diagonal block is of the form described in the preceding example). In the above example, $|\mc{I}|=k+1$, $s_1=k$, $m_1=1$. For $i \in \left\{2, \ldots, k+1 \right\}$, $s_i=1$, $m_i=r_{i-1}-1$. Covariance matrices with product symmetric group structure arise when one has a set of random variables that are \emph{partially} exchangeable.
\end{enumerate}
\end{example}

Covariance matrices of interest in this paper will be invariant with respect to a group $\mf{G}$, and hence belong to the fixed point subspace $\mc{W}_{\mf{G}}$. A fundamental operation in our paper will be the Euclidean projection of a given matrix $M$ onto  $\mc{W}_{\mf{G}}$. The next lemma relates projection to any important notion called \emph{Reynolds averaging}. Reynolds averaging operates by generating the orbit of $M$ with respect to the group, i.e. the collection of matrices $\Pi_g M \Pi_g^T$, and averaging with respect to the orbit.

\begin{lemma} \label{lemma:reynolds}
Let $\mc{P}_{\mf{G}}(\cdot)$ be the Euclidean projection operator onto the fixed point subspace $\mc{W}_{\mf{G}}$. Then
$$\mc{P}_{\mf{G}}(M)=\frac{1}{|\mathfrak{G}|} \sum_{g \in \mathfrak{G}}\Pi_g M\Pi_g^{T}.$$
\end{lemma}
\begin{proof}
Since $\mc{W}_{\mf{G}}$ is a subspace, it suffices to show that $\Tr\l \l M-\mc{P}_{\mf{G}}(M) \r \mc{P}_{\mf{G}}(M) \r=0$ (by the orthogonality principle and that fact that $\langle A, B \rangle =\Tr(AB)$ is the inner product that generates the Euclidean norm over the space of symmetric matrices). We obtain the proof via the following chain of equalities:
\begin{align*}
\Tr \l\mc{P}_{\mf{G}}(M) \mc{P}_{\mf{G}}(M) \r &= \Tr \l \frac{1}{|\mathfrak{G}|} \sum_{h \in \mathfrak{G}}\Pi_h M\Pi_h^{T} \frac{1}{|\mathfrak{G}|} \sum_{g \in \mathfrak{G}}\Pi_g M\Pi_g^{T} \r\\
&=\frac{1}{|\mathfrak{G}|^2}\sum_{h \in \mathfrak{G}}\Tr \l \Pi_h M\Pi_h^{T}  \sum_{g \in \mathfrak{G}}\Pi_g M\Pi_g^{T}  \r \\
&=\frac{1}{|\mathfrak{G}|^2}\sum_{h \in \mathfrak{G}}\Tr \l  M  \sum_{g \in \mathfrak{G}}\Pi_h^{T}\Pi_g M\Pi_g^{T} \Pi_h \r \\
&=\frac{1}{|\mathfrak{G}|}\sum_{h \in \mathfrak{G}}\Tr \l  M \mc{P}_{\mf{G}}(M)  \r \\
&=\Tr \l M \mc{P}_{\mf{G}}(M) \r.
\end{align*}
\end{proof}

\begin{remark} \label{remark:averaging}
Note that the fixed point subspace, when suitably transformed, is a set of block diagonal matrices of the form \eqref{eq:decomp} as a consequence of Schur's lemma. Projection of a matrix $M$ onto the fixed point subspace in these coordinates corresponds precisely to \emph{zeroing} components of $M$ that are off the block diagonals, following by \emph{averaging} the diagonal blocks corresponding to multiplicities of the same irreducible representation.
\end{remark}

\subsection{Group theoretic parameters of interest} \label{sec:group_parameters}
Given a covariance matrix $\Sigma$ which is invariant with respect to the action of a group $\mf{G}$, we introduce the following parameters that depend on $\Sigma$ and $\mf{G}$. Our subsequent results characterizing the rates of convergence of the estimated covariance to the true covariance will be expressed in terms of these group theoretic parameters.

Given $\Sigma$, it will often be convenient to associate informally a weighted undirected complete graph to it, so that the entries of $\Sigma$ correspond to edge weights. We will view the rows and columns of $\Sigma$ as being indexed by a set $V$ so that $\Sigma \in \R^{|V| \times |V|}$. The group $\mf{G}$ is then precisely the \emph{automorphism group} of this weighted graph. Let $(i,j) \in V \times V$ (we will call this an ``edge''). Define the \emph{edge orbit} as 
$$\orb(i,j)=\lb (g(i),g(j)) \;|\; g \in \mathfrak{G} \rb. \begin{footnote}{The pairs $(g(i),g(j))$ are to be viewed as unordered. Also, each edge is enumerated only once so that if $(g(i),g(j))=(i,j)$ then we treat them as the same edge in $\orb(i,j)$. }\end{footnote}
$$
We further define $\mc{V}(i,j)$ to be the set of nodes that appear in the edge orbit $\orb(i,j)$, i.e.
$$
\mc{V}(i,j)=\lb k \in V \;|\; \exists l \in V,  g \in \mathfrak{G} \text{ such that } (g(k),g(l)) \in \orb(i,j) \rb.
$$
We define $d_{ij}$ (called the degree of the orbit $\orb(i,j)$) to be the maximum number of times any node appears in the edge orbit $\orb(i,j)$, i.e. $d_{ij}(k)=\sum_{l \in \mc{V}(i,j)} \one \left( \exists l \in V \; | \; (k,l) \in \orb(i,j)  \right)$, and $d_{ij} = \max_{k \in \mc{V}(i,j)} d_{ij}(k).$

\begin{definition}
Let $M \in \mc{W}_{\mf{G}}$ be a matrix in the fixed point subspace.
We define the \emph{orbit parameter} $\orb$ to be the size of the smallest orbit (with respect to the action of $\mf{G}$ on M) , i.e.
$$
\orb = \min_{i,j} |\orb(i,j) |
$$
The \emph{normalized orbit parameter} $\orb_d$ is defined to be
$$
\orb_d = \min_{i,j} \frac{|\orb(i,j)|}{d_{ij}}.
$$
\end{definition}

\begin{example} \label{example:cyclic}
For a covariance matrix of the form \eqref{eq:circ} shown in Example \ref{example:1}, 
\begin{align*}
\orb(1,2)&=\left\{(1,2),(2,3),\ldots, (p-1,p),(p,1) \right\}  \\ \mc{V}(1,2)&=\left\{1, \ldots, p\right\}  
\\ d_{12}&=2 \qquad \text{(since each node appears twice in } \orb(1,2)).
\end{align*}
For this example $\orb=p$ and $\orb_d=\frac{p}{2}$.
\end{example}

\begin{example} \label{example:symm}
For a covariance matrix of the form \eqref{eq:full_sym} with the action of the symmetric group of order $p$, it is easy to check that $\orb=p$ (there are only two orbits, the first is $\orb(1,1)$ which is of size $p$ and degree $1$, and the other is  $\orb(1,2)$ of size $\frac{p(p-1)}{2}$ and degree $p-1$. Hence $\orb=p$ and $\orb_d = \frac{p}{2}$.
\end{example}

\begin{example} \label{example:ksymm}
For a covariance matrix of the from \eqref{eq:prod_sym} with the action of the product of symmetric groups, if the sizes $|T_i|=r_i$ for $i=1, \ldots, k$  then it is clear that $\orb=\min_i r_i$ and $\orb_d = \min_i \frac{r_i}{2}$.
\end{example}

In addition to these combinatorial parameters, we will also be interested in the following \emph{representation-theoretic} parameters (which are also intrinsic to the group and the fixed point subspace $\mc{W}_{\mf{G}}$). Given $\Sigma$ and its block diagonalization \eqref{eq:decomp} as described by Schur's lemma, we will be interested in the parameters $s_i$ (the sizes of the blocks) and $m_i$ (the multiplicities) for $i \in \mc{I}$. (We note that given a fixed point subspace $\mc{W}_{\mf{G}}$, multiplicities corresponding to many irreducible components may be zero, we call these the inactive components. We will only be interested in the active components. Thus the set of irreducible representations $\mc{I}$ mentioned henceforward will include only active representations).

\subsection{Symmetric Covariance Models}
Let $X \sim \mc{N}(0,\Sigma)$ be a normally distributed $\R^p$-valued random vector where the covariance $\Sigma$ is unknown. Given a small number of samples $X^{(1)}, \ldots, X^{(n)}$, we wish to learn properties of the model (for example, the covariance matrix $\Sigma$, the inverse covariance $\Theta:=\Sigma^{-1}$, the sparsity pattern of $\Sigma$ and $\Theta$). We will be interested in the \emph{high-dimensional setting} where $p$ is large and $n \ll p$, so that the number of samples is much smaller than the ambient dimension of the problem.

Throughout the paper we will make the two following assumptions regarding our model:
\begin{assumption} {\ \\ } \label{assumption:model}
\begin{enumerate}
\vspace{-10pt}
\item We assume that $\Sigma$ is \emph{invariant} with respect to a known group $\mf{G}$ so that $\Sigma \in \mc{W}_{\mf{G}}$.
\item We assume that the spectral norm of the true covariance matrix is bounded independent of the model size $p$, i.e.
$$
\|\Sigma\| \leq \psi
$$
for some constant $\psi$.
\end{enumerate}
\end{assumption}

\begin{remark}
A simple but important consequence of the $\mf{G}$-invariance of $\Sigma$ (i.e. $\Pi_g \Sigma \Pi_g^T=\Sigma$ for all $g \in \mf{G}$) is that $\Sigma_{ii}=\Sigma_{g(i),g(i)}$, and $\Sigma_{ij}=\Sigma_{g(i),g(j)}$ for all $g \in \mf{G}$.
\end{remark}

\noindent Given samples $X^{(1)}, \ldots X^{(n)}$ we define the following.
\begin{enumerate}
\item The \emph{sample Reynolds average} of the $k^{th}$ sample is:
\begin{equation} \label{eq:sample_average}
\hat{\Sigma}^{(k)}_{\mc{R}}:= \mc{P}_{\mf{G}} \l X^{(k)} (X^{(k)})^T \r=
\frac{1}{|\mathfrak{G}|} \sum_{g \in \mathfrak{G}}\Pi_g X^{(k)}{X^{(k)}}^{T}\Pi_g^{T}.
\end{equation}
\item The \emph{empirical covariance matrix} is:
\begin{equation}  \label{eq:sample_reynolds_average}
\Sigma^n:=\frac{1}{n}\sum_{k=1}^{n} X^{(k)}{X^{(k)}}^{T}.\end{equation}
\item The \emph{$\mf{G}$-empirical covariance matrix} is:
\begin{equation} \label{eq:sample_mixed_average}
\hat{\Sigma}:=\mc{P}_{\mf{G}} \l \Sigma^n \r.  
\end{equation}
\end{enumerate}
Note that the $(i,j)$ entry of the $\mf{G}$-empirical covariance matrix may be written as:
\begin{equation} \label{eq:symm_est}
\hat{\Sigma}_{ij}=\frac{1}{n|\mathfrak{G}|} \sum_{k=1}^n \sum_{g \in \mathfrak{G}} X_{g(i)}^{(k)}X_{g(j)}^{(k)},
\end{equation}
where $X^{(k)}_i$ is the $i^{th}$ component of the sample $X^{(k)}$.
We point out that \eqref{eq:symm_est} may be interpreted as averaging $X_i^{(k)}X_j^{(k)}$ with respect to the orbit $\orb(i,j)$, followed by averaging over the samples $k=1, \ldots, n$. As we will argue subsequently, due to the double averaging, the estimate of the covariance matrix obtained via the $\mf{G}$-empirical covariance matrix is a much better estimator as compared to the empirical covariance matrix. Indeed, quantification of this fact is one of the central themes of this paper.

\begin{remark}
Note that Lemma \ref{lemma:reynolds} ensures that if $M \succeq 0$, then $\mc{P}_{\mf{G}}(M) \succeq 0$ (since positive definiteness is preserved under conjugation and addition). Hence the projection of the empirical covariance matrix $\mc{P}_{\mf{G}}(\Sigma^n)$ is a meaningful estimate of $\Sigma$ in the sense that it is a valid covariance matrix.
\end{remark}

\subsubsection{Sufficient statistics}
We make a few brief remarks about the \emph{sufficient statistics} for symmetric Gaussian models. For an arbitrary (zero-mean) model, it is well-known that the sufficient statistics correspond to quadratics associated to the variance, i.e. the sufficient statistics are
$$
\phi_{ij}(x)=x_ix_j \qquad \text{for } (i, j) \in V \times V.
$$
For $\mf{G}$-invariant models, the sufficient statistics are
$$
\phi_e(x)=x_i x_j  \qquad \text{for } e=(i, j) \in (V \times V) | \mf{G},
$$
where $(V \times V) | \mf{G}$ denotes the \emph{quotient} of $ V \times V$ with respect to  $\mf{G}$ with respect to the equivalence relation $(i,j) \cong (k,l)$ if $g(i)=k$ and $g(j)=l$ for some $g \in \mf{G}$. Put more simply, the sufficient statistics correspond to quadratics  $\phi_{ij}$, and we have only one statistic per \emph{orbit} $\orb(i,j)$ (as opposed to one per pair $(i,j)$). The representative sufficient statistic from orbit $\orb(i,j)$ may be picked to be: $$\frac{1}{|\orb(i,j)|}\sum_{(i,j) \in \orb(i,j)} x_ix_j.$$

An alternate set of sufficient statistics can be obtained via Schur's lemma by observing that $Y=T^*X$ is an equivalent statistical model whose covariance is block diagonal. The sufficient statistics are $\theta_{ij}(y)= y_iy_j$ for $i,j$ belonging to the same irreducible component. Note that one only needs one set of statistics per irreducible component, and they may be expressed as above by averaging the corresponding sufficient statistics with respect to the multiplicities of each irreducible component. 

\section{Rates of convergence} \label{sec:3}
In this section we will be interested in the statistical properties of the estimator $\hat{\Sigma}:= \mc{P}_{\mf{G}}(\Sigma^n)$. Note that given samples $X^{(1)}, \ldots X^{(n)}$, $\Ex(\hat{\Sigma})=\Sigma$. In this section, we will be interested in establishing \emph{rates of convergence} of $\hat{\Sigma}$ in different norms, i.e. we will provide upper bounds on
$
\Prob \l( \|\hat{\Sigma} -\Sigma \|_p \geq \delta \r,
$
where $\| \cdot \|_p$ are the Frobenius, spectral and $\ell_{\infty}$ norms. The rates presented here hold for the covariance estimation of any sub-Gaussian random vector $X$. Though we present them in the context of normally distributed random vectors, identical analysis holds for sub-Gaussians more generally.

\subsection{Spectral norm and Frobenius norm rates}

To obtain our rate results, we use the following standard result regarding spectral norm rates for an arbitrary normal model. 
\begin{lemma} \label{lemma:standard_rate}
Let $X \sim \mc{N}(0,\Sigma^*)$ be a normal model with $\Sigma^* \in \R^{p \times p}$. Let $\psi=\| \Sigma^* \|$. Given $\delta >0$ with $\delta \leq 8 \psi$, let the number of samples satisfy $n \geq \frac{64p \psi^2}{\delta^2}$. Let $\Sigma^n$ be the empirical covariance matrix formed from the $n$ samples. Then we have
$$
\Prob \l  \| \Sigma^n - \Sigma^* \| \geq \delta \r \leq 2 \exp \l \frac{-n \delta^2}{128 \psi^2}\r.
$$
\end{lemma}
\begin{proof}
This is a standard result, a proof follows from \cite[Theorem II.13]{DavidsonSzarek}.
\end{proof}

\begin{theorem} \label{theorem:spectral_rate}
Given an invariant covariance matrix $\Sigma$ with $\psi = \|\Sigma\|$ and a sample covariance $\Sigma^n$ formed from $n$ i.i.d. samples drawn from a Gaussian distribution according to covariance $\Sigma$, we have that
\begin{equation*}
\Prob[\|(\Sigma - \mc{P}_{\mathfrak{G}}(\Sigma^n))\| \geq \delta] \leq |\mathcal{I}| \exp \l -\frac{n \min_{i \in \mathcal{I}} m_i \delta^2}{128 \psi^2} \r
\end{equation*}
for $n \geq \max_{i \in \mathcal{I}} \frac{64 s_i \psi^2}{m_i \delta^2}$ and for $\delta \leq 8 \psi$.
\end{theorem}
\begin{proof}
 Due to the unitary invariance of the spectral norm, we may assume that $\Sigma$ is put in the suitable block-diagonal coordinates based on Schur's representation lemma.  Letting $\mathcal{I}$ denote the indices of the set of active irreducible representations, using Lemma \ref{lemma:standard_rate} we have for each $i \in \mathcal{I}$ that
\begin{equation*}
\Prob[\|(\Sigma - \mc{P}_{\mathfrak{G}}(\Sigma^n))_i\| \geq \delta] \leq 2 \exp\l-\frac{n m_i \delta^2}{128 \psi^2}\r
\end{equation*}
for $\delta \leq 8 \psi$ and for $n m_i \geq \frac{64 s_i \psi^2}{\delta^2}$.  Note that we get the extra factor of $m_i$ (the multiplicity of the representation indexed by $i$) since blocks corresponding to the same irreducible representation but different multiplicity are independent and identically distributed. Hence we average over the $m_i$ blocks of size $s_i$ (see Remark \ref{remark:averaging}).  Consequently we have by the union bound that
\begin{equation*}
\Prob[\|(\Sigma - \mc{P}_{\mathfrak{G}}(\Sigma^n))\| \geq \delta] \leq |\mathcal{I}| \exp\l-\frac{n \min_{i \in \mathcal{I}} m_i \delta^2}{128 \psi^2}\r
\end{equation*}
for $n \geq \max_{i \in \mathcal{I}} \frac{64 s_i \psi^2}{m_i \delta^2}$. 
\end{proof}
\begin{theorem}
Given an invariant covariance matrix $\Sigma$ with $\psi = \|\Sigma\|_2$ and a sample covariance $\Sigma^n$ formed from $n$ i.i.d. samples drawn from a Gaussian distribution according to covariance $\Sigma$, we have that
\begin{equation*}
\Prob[\|(\Sigma - \mc{P}_{\mathfrak{G}}(\Sigma^n))\|_F \geq \delta \sqrt{p}] \leq |\mathcal{I}| \exp\l-\frac{n \min_{i \in \mathcal{I}} m_i \delta^2}{128 \psi^2}\r
\end{equation*}
for $n \geq \max_{i \in \mathcal{I}} \frac{64 s_i \psi^2}{m_i \delta^2}$ and for $\delta \leq 8 \psi$.
\end{theorem}

\begin{proof} The proof of this result follows directly from that of the previous theorem by noting that $\|\cdot\|_F \leq \sqrt{p} \|\cdot\|$ for $p \times p$ matrices. \end{proof}

As a direct corollary of the preceding theorems, we get the following sample complexity results.
\begin{corollary}
Assume that $\psi$ is a constant (independent of $p, n$). The sample complexity of obtaining $\| \Sigma - \hat{\Sigma} \| \leq \delta$ for $\delta \leq 8 \psi$ with probability exceeding $1-\frac{1}{p^c}$ for some constant $c$ is given by:
\begin{align*}
n \geq C_1 \max \left\{  \max_{i \in \mc{I}} \frac{s_i}{m_i \delta^2}, \max_{i \in \mc{I}} \frac{\log p}{m_i \delta^2}\right\},
\end{align*}
for some constant $C_1$.

\noindent The sample complexity of obtaining $\| \Sigma - \hat{\Sigma} \|_F \leq \delta$ for $\delta \leq 8 \psi$ with probability exceeding $1-\frac{1}{p^c}$ for some constant $c$ is given by:
\begin{align*}
n \geq C_2 \max \left\{  \max_{i \in \mc{I}} \frac{s_i p}{m_i \delta^2}, \max_{i \in \mc{I}} \frac{p \log p}{m_i \delta^2}\right\},
\end{align*}
for some constant $C_2$.
\end{corollary}

\begin{remark}
It is instructive to compare the sample complexity described by Lemma \ref{lemma:standard_rate} versus the rates obtained from Theorem \ref{theorem:spectral_rate} for the case of the covariance matrices invariant with respect to the cyclic, the symmetric, and the product of symmetric groups. Note that the spectral norm sample complexity obtained from Lemma \ref{lemma:standard_rate} is $n \geq \frac{Cp}{\delta^2}$ for some constant $C$. As explained in Example \ref{example:1}, for the cyclic group case $s_i=m_i=1$ for all $i$ so that the sample complexity is $n \geq \frac{C \log p}{\delta^2}$. For the symmetric group $s_1=1, m_1=1$ and $s_2=1, m_2=p-1$ so that the sample complexity becomes $n \geq \frac{C \log p}{\delta^2}$. For the case of product of $k$ symmetric groups, the sample complexity is $n \geq \frac{C \max \left\{k , \log p \right\}}{\delta^2}$. Hence in all three cases, projection onto the fixed point subspace delivers a dramatic improvement in the sample complexity.
\end{remark}

\subsection{$\ell_{\infty}$ norm rates}
%

In this section we examine the $\ell_\infty$ rates of convergence of the estimate $$\hat{\Sigma}:=\mc{P}_{\mf{G}} \l \Sigma^n\r.$$ In this section, as elsewhere in the paper, we will assume that our covariance model $\Sigma$ has bounded spectral norm, i.e. $\| \Sigma \|$ is bounded independent on $p$ and $n$. It is possible to make the dependence on this quantity more explicit at the expense of more notation in a straightforward manner, for simplicity and clarity we omit doing so.
\begin{theorem} \label{theorem:rate}
Let $\orb(i,j)$ be the orbit of the element $(i,j)$, let its degree be $d_{ij}$. Then there are constants $C_1, C_2$ such that for all $i,j$
$$
\Prob \l |\hat{\Sigma}_{ij}-\Sigma_{ij}|> t \r \leq \max \left\{\exp \l -C_1 n|\orb(i,j)|t^2 \r , \exp \l -C_2 \frac{n|\orb(i,j)|}{d_{ij}}t \r  \right\}.
$$
\end{theorem}
\noindent The proof is provided in the appendix.
\\ \\
Recall that $\orb := \min_{i,j} | \orb(i,j)|$ denotes the size of the smallest edge orbit, and $\orb_d:=\min_{i,j} \frac{|\orb(i,j)|}{d_{ij}}$, and let $n_{\orb}$ denote the number of \emph{distinct} (i.e. inequivalent) edge orbits $\orb(i,j)$. Note that $n_\orb \leq p^2$ trivially.


\begin{corollary} \label{cor:linf_rate1}
We have,
$$
\Prob \l \max_{i,j} |\hat{\Sigma}_{ij}-\Sigma_{ij}|> t \r \leq  n_\orb\max\left\{\exp \l -C_1 n\orb t^2 \r , \exp \l -C_2 n \orb_d t \r  \right\}.
$$
\end{corollary}
\begin{proof}
From Theorem \ref{theorem:rate}, and by the definition of $\orb$ and $\orb_d$ we have for each $(i,j)$:
$$
\Prob \l |\hat{\Sigma}_{ij}-\Sigma_{ij}|> t \r \leq \max \left\{\exp \l -C_1 n\orb t^2 \r , \exp \l -C_2  n \orb_d t \r  \right\}.
$$
By taking union bound over the distinct orbits we obtain the required result.
\end{proof}
\begin{corollary}\label{cor:linf_rate2}
We have the following sample complexity results:
\begin{align}
|\Sigma_{ij}-\hat{\Sigma}_{ij}| \leq O_P \left( \max \left\{ \sqrt{\frac{\log p}{n|\orb(i,j)|}}, \frac{d_{ij}\log p}{n|\orb(i,j)|}\right\} \right)  
\end{align}
For $|\Sigma_{ij}-\hat{\Sigma}_{ij}| \leq \delta$ with probability exceeding $1-\frac{1}{p^c}$ for some constant $c$, we need
$$
n \geq C_3 \max \left\{ \frac{\log p}{\delta^2|\orb(i,j)|}, \frac{d_{ij}\log p}{\delta|\orb(i,j)|}\right\},
$$
for some constant $C_3$.
\end{corollary}
\begin{proof}
This is a simple consequence of Theorem \ref{theorem:rate}.
\end{proof}
\begin{corollary}\label{cor:linf_rate3}
We have the following sample complexity results for the $\ell_\infty$ norm:
\begin{align}
\max_{i,j}|\Sigma_{ij}-\hat{\Sigma}_{ij}| \leq O_P \left( \max \left\{ \sqrt{\frac{\log p}{n\orb}}, \frac{\log p}{n\orb_d}\right\} \right)  
\end{align}
To obtain an error bound $\max_{i,j}|\Sigma_{ij}-\hat{\Sigma}_{ij}| \leq \delta$ with probability exceeding $1-\frac{1}{p^c}$ for some constant $c$, we need
$$
n \geq C_3 \max \left\{ \frac{\log p}{\delta^2\orb}, \frac{\log p}{\delta\orb_d}\right\},
$$
for some constant $C_3$.
\end{corollary}
\begin{proof}
This is a simple consequence of Corollary \ref{cor:linf_rate1}.
\end{proof}
We remind the reader that the group-theoretic parameters $\orb$ and $\orb_d$ were introduced in Section \ref{sec:group_parameters}. These parameters roughly capture the sizes of the orbits in the group. Since samples corresponding to the same orbit are statistically equivalent (though the information is not $\iid$), we may expect to reuse this information to get statistical gains. Corollary \ref{cor:linf_rate3} makes these gains explicit, indeed the sample complexity is smaller by a factor corresponding to the size of the orbit.  
\begin{remark}
It is again instructive to contrast the standard $\ell_\infty$ rates versus the improved rates derived in this section for symmetric models. The standard $\ell_\infty$ rates require a sample complexity of  $n \geq \frac{C \log p}{\delta^2}$ for $\delta =o(1)$. For the cyclic and symmetric group, our rates imply a sample complexity of $n \geq \frac{C \log p}{\delta^2 p}$ samples. (Similarly for the case of the product of symmetric groups our sample complexity is $n \geq \frac{C \log p}{\delta^2 \min_i r_i}$.) We point out that in the case of models invariant with respect to the symmetric group, the overall $ \ell_\infty$ rate is governed by the rate of convergence of the diagonal entries of $\Sigma$ corresponding to the orbit $\orb(1,1)$. We point out that the orbit $\orb(1,2)$ is much larger, and the sample complexity for estimating $\Sigma_{12}$ is $n \geq C \max \l \frac{ \log p}{\delta^2 p(p-1)}, \frac{C \log p}{\delta p} \r$.
\end{remark}


\section{Sparse Model Selection} \label{sec:sparse_model_selection}

\subsection{Covariance Selection} \label{sec:cov}
Estimation of covariance matrices in a high dimensional setting is an important problem in statistical analysis since it is an essential prerequisite for data analysis procedures such as dimension reduction via PCA, classification, and testing independence relations. In the high dimensional setting (which arise for example from the discretization of Markov random fields), the dimension of the problem may be larger than the number of samples available. In this setting it is well-known that the empirical covariance matrix is a poor estimator for the covariance matrix \cite{BickelLevina2}. There is a growing body of literature on understanding regularization procedures for covariance matrices that possess specific structure. For example, if the covariance matrix known to be sparse, Bickel and Levina \cite{BickelLevina2} show that \emph{thresholding} the empirical covariance matrix is both a natural and statistically meaningful means of regularization. Indeed they provide precise rates of convergence of the thresholded estimate to the true covariance matrix.

In this section we analyze sparse covariance matrices arising from Gaussian models with symmetries. We show that if the covariance matrix is known to be invariant under the action of $\mathfrak{G}$ and is sparse, a statistically meaningful way of regularizing the empirical covariance matrix is to first project it onto the fixed point subspace, followed by thresholding. Doing so provides a faster rate of convergence and reduces the sample complexity by a factor that depends on the group  (more specifically on the size of the smallest orbit and the orbit degree).

More precisely, let $X \sim \mathcal{N}(0,\Sigma)$, and let $\Sigma$ be invariant under the action of the group $\mathfrak{G}$. Furthermore, we adopt a model of sparsity on $\Sigma$ that is somewhat more general (Bickel and Levina \cite{BickelLevina2}) since it encompasses the usual notion of sparsity as well as power-law decay in the coefficients of $\Sigma$ as follows.
Define the set of approximately sparse covariance matrices via a \emph{uniformity class} as:
\begin{equation} \label{eq:uniformity_class}
\mc{U}(q)=\left\{ \Sigma \; | \; \Sigma_{ii}\leq M, \sum_{j=1}^{p}|\Sigma_{ij}|^{q} \leq c_{0}(p), \; \text{ for all } i \right\}
\end{equation}
for $0<q<1$. If $q=0$ we have a set of sparse matrices as follows:
\begin{equation} \label{eq:uniformity_class0}
\mc{U}(0)=\left\{ \Sigma \; | \; \Sigma_{ii}\leq M, \sum_{j=1}^{p}\one(\Sigma_{ij} \neq 0) \leq c_{0}(p), \; \text{ for all } i \right\}.
\end{equation}
Here we think of $c_0(p)$ as a parameter that controls the sparsity level of $\Sigma$, and is to be thought of as sublinear in $p$.
Given samples $X_1, \ldots, X_n$ from this distribution, recall that $\Sigma^n$ is the empirical covariance matrix and the $\mf{G}$-empirical covariance matrix is $\hat{\Sigma}:=\mc{P}_{\mf{G}} \l \Sigma^n\r$. 

Recall that $\one(A)$ denote the indicator function for the event $A$ and define the thresholding operator (at threshold $s$) by
\begin{equation}
\mc{T}_{s}(M)=\left[ M_{ij}\one(|M_{ij}| \geq s)\right],
\end{equation}
so that the $(i,j)$ entry of $M$ is retained if it is larger than $s$ and is zero otherwise.
We propose the following estimator for the covariance:
$$
\Sigma_{\text{est}}=\mc{T}_t\l\mc{P}_{\mf{G}}\l\Sigma^n \r \r
$$
for a suitable choice of threshold parameter $t$ to be specified in the subsequent theorem.
As in \cite{BickelLevina2}, we note that the thresholding operator preserves symmetry and under mild assumptions on the eigenvalues, preserves positive-definiteness. Indeed, the next theorem provides a rate of convergence of our thresholded estimate to the true covariance in spectral norm. As a consequence, if the smallest eigenvalue of $\Sigma$ is bounded away from $0$ (independent of $p$ and $n$), the theorem also establishes that with high probability the thresholded estimate is positive definite. 
\begin{theorem} \label{theorem:covariance}
Suppose $\Sigma \in \mc{U}(q)$ satisfies Assumtion \ref{assumption:model}. Choose the threshold parameter
$$
t=M^{\prime} \max \left\{ \sqrt{\frac{\log p}{n\orb}}, \frac{\log p}{n\orb_d}\right\}
$$
for a constant  $M^{\prime}$. 
Then uniformly on $\mathcal{U}(q)$
$$
\|\mc{T}_{t} \l  \mc{P}_{\mf{G}} \l \Sigma^n \r \r -\Sigma \|=O_{P}\l c_0(p)  \max \left\{ \sqrt{\frac{\log p}{n\orb}}, \frac{\log p}{n\orb_d}\right\}^{1-q} \r.
$$
\end{theorem}

\begin{remark}
Note that the sample complexity gains observed in the $\ell_{\infty}$ norms translate directly to gains for sparse covariance selection. Indeed, in the absence of symmetry the deviations $\delta=O_{P}\l c_0(p) \l \frac{\log p}{n} \r^{\frac{1-q}{2}}  \r$. In the case of models that are sparse and invariant with respect to the cyclic group, we see $\delta=O_{P}\l c_0(p) \l \frac{\log p}{np} \r ^{\frac{1-q}{2}}  \r$. Moreover, for invariance with respect to the symmetric group, we have $\delta=O_{P}\l c_0(p) \l \frac{\log p}{n p} \r^{{1-q}} \r$. Hence we again observe sample complexity gains corresponding to the orbit parameters.
\end{remark}

\begin{remark}
We note that as in \cite{BickelLevina1}, sample complexity results can be established in the Frobenius norm, various other tail conditions, and for models with alternative forms of sparsity \cite{BickelLevina2} arising from banding, tapering, etc. We omit a detailed discussion of the same in this paper.
\end{remark}

\subsection{Inverse Covariance Selection} \label{sec:inv_cov}
In the preceeding section we studied a means of regularizing structured covariance matrices which were sparse and symmetric with respect to $\mathfrak{G}$. In many applications, it is much more suitable to study sparsity in the concentration (i.e. inverse covariance matrix). These are often best described using the terminology of \emph{graphical models} \cite{Lauritzen_book}.

Formally we assume that the inverse covariance of $\Theta^* := \Sigma^{-1}$ is sparse. We will associate a weighted undirected graph $\mc{{G}}=(V,E)$ with $p$ nodes and edge weights $\Theta^{*}_{ij}$ so that $\Theta^*_{ij} \neq 0$ if and only if $(i,j) \in E$. In the context of Gaussian graphical models, the edges in the graph have a particularly appealing interpretation which makes them suitable in the modeling of Gauss-Markov random fields. Indeed, the edges represent conditional independence relations (i.e. $X_i$ and $X_j$ are conditionally independent of a set of nodes $X_S$ for $S \subseteq [p]$ if and only if $S$ is a separator of $i$ and $j$ in the graph $\mc{G}$).  

 A fundamental problem in graphical models learning is that of inferring the structure of the graph $\mc{G}$ and the corresponding non-zero entries of $\Theta^*$ from measurements (see \cite{RWRY} and the references therein); this problem is called the \emph{graphical model selection problem}. In the high-dimensional setup, one wishes to do so from a small number of measurements $n \ll p$. This has been shown \cite{Meinshausen,RWRY} to be possible provided the graph has additional structure, and is simple in the sense that it has low degree.

Ravikumar et al. \cite{RWRY} propose a convex optimization based procedure to obtain consistent model selection of $\Theta^*$ in the high-dimensional regime. Specifically, if $\Sigma^n$ is the empirical covariance matrix, they propose selecting $\Theta^*$ as the solution to the following $\ell_1$ regularized $\log \det$ problem:
\begin{equation} \label{eq:RWRY1}
\hat{\Theta}:= \underset{\Theta \in \mathcal{S}^{p}_{++}}{\arg \min} \; \Tr(\Sigma^n\Theta) -\log \det (\Theta) + \mu_{n} \|\Theta \noff.
\end{equation}
This approach has several appealing properties. Firstly the above optimization problem is \emph{convex}, i.e. both the optimization objective and the constraint set are convex. Consequently, this formulation is amenable to efficient computation, since one can apply standard numerical procedures to efficiently solve for the global optimum of the above convex optimization problem.

Secondly, the above formulation is statistically meaningful. Indeed if $\mu_n=0$ one recovers the standard maximum entropy formulation of the problem, and the $\ell_1$ penalty on the off-diagonal terms in $\Theta$ can be viewed as a regularization term which encourages sparsity. In \cite{RWRY}, the authors prove that the solution to problem \eqref{eq:RWRY1} has good statistical properties (under reasonable statistical assumptions). The model selection can be shown to be consistent with good rates of convergence (see \cite[Theorem 1, Theorem 2]{RWRY}) when the regularizer $\mu_n$ is chosen judiciously. Importantly, these rates are governed by the \emph{degree} of the graph $\mc{G}$ and only the logarithm of the ambient dimension $p$.

In this section, we will study in a similar setup the problem of Gaussian graphical model selection in which the graph $\mc{G}$ has low degree and also a nontrivial automorphism group. Recall that $\Sigma \in \mf{G}$ implies that $\Theta^* \in \mf{G}$, hence $\mathfrak{G}$ is the \emph{automorphism group} of the graph $\mc{G}$. We will assume that the group $\mathfrak{G}$ is non-trivial and known, where as the graph $\mc{G}$ is unknown. We examine how knowledge of the structure of $\mathfrak{G}$ can be exploited for the model selection problem using i.i.d. samples $X^{(1)}, \ldots, X^{(n)}$ drawn from $\mc{N}(0,(\Theta^*)^{-1})$ to obtain statistical gains in terms of the sample complexity. 
%
%

We propose exploiting the known symmetry structure by solving the modified log-determinant program with subspace constraints:
\begin{equation} \label{eq:opt1}
\hat{\Theta}:= \underset{\Theta \in \mathcal{S}^{p}_{++} \cap \mc{W}_{\mathfrak{G}}}{\arg \min} \; \Tr(\Sigma^n\Theta) -\log \det (\Theta) + \mu_{n} \|\Theta \noff.
\end{equation}
The following lemma shows that solving the above problem is equivalent to symmetrizing the empirical covariance matrix directly and then solving the standard formulation.
\begin{lemma}
Let $\hat{\Sigma}:=\mc{P}_{\mf{G}}\l \Sigma^n\r$. The optimization problem 
\begin{equation} \label{eq:opt2}
\underset{\Theta \in \mathcal{S}^{p}_{++}}{\arg \min} \; \Tr(\hat{\Sigma}\Theta) -\log \det (\Theta) + \mu_{n} \|\Theta \noff
\end{equation}
has the same solution as \eqref{eq:opt1}.
\end{lemma}
\begin{proof}
Let $\Theta_0$ be the optimal solution of \eqref{eq:opt1}. Let $f_s(\Theta)= \Tr(\Sigma^n\Theta) -\log \det (\Theta) + \mu_{n} \|\Theta \noff$. Then by definition $f_s(\Theta_0)=f_s(\Pi_g \Theta \Pi_g^{T})$ for all $g \in \mathfrak{G}$.  Since $f_s(\Theta)$ is convex, by Jensen's inequality we have:
$$
f_s \l \frac{1}{|\mathfrak{G}|} \sum _{g\in \mathfrak{G}} \Pi_g \Theta_0 \Pi_g^{T} \r \leq f_s(\Theta_0).
$$
However, $\frac{1}{|\mathfrak{G}|} \sum _{g\in \mathfrak{G}} \Pi_g \Theta_0 \Pi_g^{T}$ is feasible for the problem \eqref{eq:opt1}, hence 
$$
f_s \l \frac{1}{|\mathfrak{G}|} \sum _{g\in \mathfrak{G}} \Pi_g \Theta_0 \Pi_g^{T} \r = f_s(\Theta_0).
$$
Next, note that $$\Tr \l \Sigma^n \frac{1}{|\mathfrak{G}|} \sum _{g\in \mathfrak{G}} \Pi_g \Theta_0 \Pi_g^{T} \r = \Tr \l \hat{\Sigma} \Theta_0 \r.$$ 
Hence the problem \eqref{eq:opt1} has the same optimal value as:
\begin{equation} \label{eq:opt3}
\underset{\Theta \in \mathcal{S}^{p}_{++}\cap \mc{W}_\mathfrak{G}}{\arg \min} \; \Tr(\hat{\Sigma}\Theta) -\log \det (\Theta) + \mu_{n} \|\Theta \noff.
\end{equation}
(Both the optimization problems are strictly convex, and hence have the same unique solution \cite[Appendix A]{RWRY}).

Now we note that $\Tr(\hat{\Sigma} \Theta)=\Tr(\Pi_g\hat{\Sigma}\Pi_g^{T} \Theta)=\Tr(\hat{\Sigma} \Pi_g\Theta \Pi_g^{T})$ for all $g \in \mathfrak{G}$. Moreover $\log \det (\cdot)$ is unitarily invariant. Lastly, note that $\Pi_g \Theta \Pi_g^{T}$ for any permutation matrix $\Pi_g$ has the effect of permuting the rows and columns of $\Theta$, which does not affect the $\| \cdot \noff$ norm. Hence the objective function of \eqref{eq:opt2} is invariant with respect to $\mathfrak{G}$ and is convex. Since the constraint set of \eqref{eq:opt2} (the positive definite cone $\mathcal{S}^{p}_{++}$) is also invariant with respect to $\mathfrak{G}$ and convex, by a standard convexity argument it is possible to show that its solution must belong to the fixed point subspace. Thus the subspace constraint in \eqref{eq:opt3} is redundant, and may be dropped. 
\end{proof}
Before we proceed to state the statistical properties of the optimal solution to \eqref{eq:opt1}, we must make some standard statistical assumptions. 
We denote $\Sigma$ and $\Theta^{*}$ to be the covariance and concentration matrix respectively for the true model. We denote by $S$ the set of indices $(i,j)$ (including the diagonal entries) that are non-zero in $\Theta^{*}$, and denote by $S^c$ the complement of this set. For any two index sets $T, T^\p$, and define the matrix $M_{T,T^\p}$ to be the submatrix of $M$ with rows and columns indexed by $T, T^{\p}$ respectively.
Finally we define $\Gamma^*:=\Theta^{*-1} \kron \Theta^{*-1}$, where $\kron$ denotes the Kronecker product. 

We define the following quantities:
\begin{align}
\kappa_{\Sigma^*}= \| \Sigma \|_{\infty, \infty} \qquad \kappa_{\Gamma^*}=\|\l\Gamma^*_{SS}\r^{-1} \|_{\infty, \infty}.
\end{align}
For simplicity, we will assume that that these quantities are \emph{constants} in our asymptotic setting, i.e. they are bounded independent of $p$ and $n$ (although it is straightforward to make the dependence more explicit). Finally we make the following \emph{incoherence} assumption which is also standard for such statistical models \cite[Assumption 1]{RWRY}:
\begin{assumption} \label{ass:1}
There exists some $\alpha \in (0,1]$ such that
$$
\max_{e \in S^c} \| \Gamma^*_{eS} \l\Gamma^*_{SS}\r^{-1}\|_{1} \leq 1-\alpha.
$$
\end{assumption}

We define the \emph{degree} of the graphical model to be 
$$
d_{\mc{G}}:=\max_{i=1, \ldots, p} \left\{ \sum_{j=1}^{p} \one_{\left\{\Theta^{*}_{ij}\neq 0\right\}} \right\}.
$$ 
Note that this precisely corresponds to the degree of the graph $\mc{G}$.


\begin{theorem} \label{theorem:RWRY1}
Consider a symmetric Gaussian graphical model satisfying Assumption \ref{assumption:model} and Assumption \ref{ass:1}. Let $\hat{\Theta}$ be the unique solution to the problem log-determinant program \eqref{eq:opt1} with regularization parameter 
$$\lambda_n= 
 M_1 \max \left\{ \sqrt{\frac{\log p}{n\orb}}, \frac{\log p}{n\orb_d}\right\} 
$$ for some constant $M_1$. Suppose the sample size is bounded by
$
n> \max (n_1, n_2), 
$
\begin{align*}
n_1= M_2 \frac{d_{\mc{G}}^2\log p}{\orb} \qquad n_2=  M_3\frac{d_{\mc{G}}\log p}{\orb_d}
\end{align*}
for some constants $M_2$ and $M_3$.
Then with probability exceeding $1 - \frac{1}{p^{\tau}}$ for some constant $\tau>0$ we have:
\begin{enumerate}
\item The estimate $\hat{\Theta}$ satisfies:
$$
\| \hat{\Theta}-\Theta^* \|_{\infty} \leq  M_0\max \left\{ \sqrt{\frac{\log p}{n\orb}}, \frac{\log p}{n\orb_d}\right\},
$$
for some constant $M_0$.
\item It specifies an edge set $E(\hat{\Theta})$ that is a subset of the true edge set $E(\Theta^*)$, and includes all edges $(i,j)$ with $|\Theta^{*}_{ij}|>M_4\max \left\{ \sqrt{\frac{\log p}{n\orb}}, \frac{\log p}{n\orb_d}\right\}$ for some constant $M_4$.
\end{enumerate}
\end{theorem}
\begin{proof}
Is a direct consequence of Corollary \ref{cor:linf_rate2} and \cite[Theorem 1]{RWRY}.
\end{proof}
\begin{remark}
We note that the constants $M_0, M_1, M_2, M_3, M_4$ depend on the parameters $ \| \Sigma \|, \alpha, \kappa_{\Sigma^*}, \kappa_{\Gamma^*}, \tau$ in a way that can be made explicit, as has been done in \cite{RWRY}.
\end{remark}
A more refined version of the second part of the preceding theorem that is analogous to \cite[Theorem 2]{RWRY}, and that establishes model selection consistency can also be established using our $\ell_\infty$ bounds, though we omit the statement of a precise result here.
\begin{remark}
Note that sample complexity gains in the $\ell_\infty$ norm rates are inherited for sparse inverse covariance selection. Indeed, the standard sample complexity requires $n \geq C d_{\mc{G}}^2 \log p$ samples to get an accuracy of
$$
\| \hat{\Theta}-\Theta^* \|_{\infty} \leq  M  \sqrt{\frac{\log p}{n}}, 
$$
we have significant gains for models with symmetry. 
For the cyclic group, we need $n \geq C d_{\mc{G}}^2 \frac{\log p}{p}$ to obtain
$$
\| \hat{\Theta}-\Theta^* \|_{\infty} \leq  M  \sqrt{\frac{\log p}{n p}}.
$$
For the symmetric group also we need $n \geq C d_{\mc{G}} \frac{\log p}{p}$ to obtain
$$
\| \hat{\Theta}-\Theta^* \|_{\infty} \leq  M  \sqrt{\frac{\log p}{n p}}.
$$
\end{remark}

\subsubsection{Extension to Latent Variable Graphical Model Selection}
The results presented in \cite{RWRY} have been extended in \cite{venkat_latent} to deal with the
setting in which one may not observe samples of all the relevant variables.  In
contrast to graphical model selection in which one seeks a sparse approximation
to a concentration matrix, \cite{venkat_latent} proposed a convex optimization method to
approximate a concentration matrix as the sum of a sparse matrix and a low-rank
matrix.  The low-rank component captures the effect of marginalization over a
few latent variables, and the sparse component specifies the conditional
graphical model among the observed variables conditioned on the latent
variables.  The main result in \cite{venkat_latent} is that consistent model selection in the
latent-variable setting (i.e., decomposition into sparse and low-rank components
with the components having the same support/rank as the true underlying
components) via tractable convex optimization is possible under suitable
identifiability conditions even when the number of latent variables and the
number of samples of the observed variables are on the same order as the number
of observed variables.  The sample complexity in these consistency results can
be improved by exploiting symmetry in the underlying model in a similar manner
to our presentation above.  To consider just one simple example, if the
marginal covariance matrix of the observed variables is invariant with respect
to the cyclic group, then consistent latent-variable model selection using the
convex programming estimator of \cite{venkat_latent} and by exploiting symmetry is possible
(under suitable identifiability conditions) when the number of latent variables
is on the same order as the number of observed variables, but with the number of
samples of the observed variables growing only logarithmically in the number of
observed variables.  This improvement in the rates follows directly from our
analysis in Section \ref{sec:3}.

\section{Computational Aspects}\label{sec:computational_aspects}
In the preceding sections we have argued that knowledge of symmetry in a statistical model can lead to significant statistical gains in terms of sample complexity. To do so, one must exploit knowledge regarding symmetry judiciously by regularizing the model appropriately. We will now investigate the computational efficiency associated to symmetry regularization.

\subsection{Projection onto the fixed-point subspace}
A fundamental computational step in exploiting symmetry in statistical models is to project the empirical covariance matrix $\Sigma^n$ onto the fixed point subspace $\mc{W}_\mf{G}$, i.e. computing
$$
\hat{\Sigma}=\mc{P}_{\mf{G}}(\Sigma^n).
$$
As argued in Section \ref{sec:3}, this estimator is a far more accurate estimator of the covariance as compared to the empirical covariance matrix. Moreover, in the selection of sparse covariance models with symmetry, our estimator is of the form:
$$
\Sigma_\text{est}=\mc{T}_t\l \mc{P}_{\mf{G}}(\Sigma^n)\r
$$
where $t$ is an appropriate threshold level as described in Section \ref{sec:cov}. In the selection of sparse graphical models as described in Section \ref{sec:inv_cov}, one needs to solve problem \eqref{eq:opt2}, which requires the computation of $\hat{\Sigma}=\mc{P}_{\mf{G}}(\Sigma^n)$.

Conceptually the projection can be computed in numerous ways. If the group is of small order, it is possible to compute the projection using Lemma \ref{lemma:reynolds}. If the group has a small number of generators (for example the cyclic group) $\mf{G}_{\text{gen}} \subseteq \mf{G}$, then the fixed point subspace can be expressed as $$\mc{W}_{\mf{G}}=\left\{ M \; | \; \rho(g)M=M\rho(g) \; \forall g \in \mf{G}_{\text{gen}} \right\},$$ and thus projections can be computed by solving the constrained least squares problem
$$
\mc{P}_{\mf{G}}(\Sigma^n)=\underset{M \; : \;\rho(g)M=M\rho(g) \; \forall g \in \mf{G}_{\text{gen}}}{\arg \min} \|M - \Sigma^n \|_F^2. 
$$
One need only specify one matrix equality constraint corresponding to each generator of the group (the constraints corresponding to other group elements are implied by the generators), so that for groups such as abelian groups this assumes an especially simple form.

Lastly, if the diagonalizing matrix $T$ corresponding to the irreducible representations from Lemma \ref{lemma:schur} are available, the projections are especially easy to compute. To compute $\mc{P}_{\mc{W}_\mf{G}}(\Sigma^n)$ one can simply compute $\Phi:=T^* \Sigma^n T$. One can truncate the off-diagonal blocks of $\Omega$ and average with respect to the multiplicities as described by Remark \ref{remark:averaging} to obtain $\Phi \p$. Finally, one obtains $\mc{P}_{\mc{W}_\mf{G}}(\Sigma^n)=T \Phi \p  T^*$.

\subsection{Solution of the $\ell_1$ regularized $\log \det$ problem}
Given a Gaussian graphical model which is symmetric with respect to the group $\mf{G}$, we examined in Section \ref{sec:inv_cov} how one can reliably recover the model by solving the optimization problem \eqref{eq:opt1} (or equivalently \eqref{eq:opt2}). This optimization problem is convex and expressible as a $\log \det$ conic program. An interesting consequence of symmetry reduction techniques that are now well-understood in optimization \cite{gatermann_parrilo,Vallentin1} is that by exploiting the group invariance of the objective function and the constraint set one can reduce the size of \eqref{eq:opt2}. Let $\Gamma:=T^* \Theta T$, and $S:=T^* \hat{\Sigma} T$. Then  \eqref{eq:opt2} reduces to:
\begin{align*}
\underset{\Gamma}{\min}  \qquad & \Tr \l S \Gamma \r - \log \det \l \Gamma\r + \mu_n \|T \Gamma T^* \noff \\
\text{subject to:} \qquad & \Gamma \succ 0.
\end{align*} 
 

In the resulting conic program in the transformed domain, the optimization variable $\Gamma$ is block diagonal, so that there is a reduction in the problem size. Furthermore, many iterative algorithms require linear algebraic operations (such as matrix inversion). These steps can be sped up for block diagonal matrices, and the complexity is governed by the size of the largest irreducible component as given by Schur's lemma.
%

In a high dimensional setup where the problem size $p$ is large, it is often prohibitively expensive to implement an interior point method to solve the problem due to the computational burden of the Newton step \cite{boyd_book}. Moreover, the above formulation is inefficient due to the introduction of additional variables $S_{ij}$. For these reasons, it is far more desirable to implement an \emph{accelerated first order method} that uses the special structure of the objective function. One such method described by d'Aspremont et al. \cite[pp. 4-5]{daspremont}, provides an efficient iterative algorithm with low memory requirements. A careful examination of their algorithm reveals that the iterations are $\mf{G}$-invariant, i.e. each iterate of the algorithm lies on the fixed-point subspace. Moreover, the computational bottlenecks in the implementation of the algorithm correspond to inversion and eigenvalue decompositions of matrices that belong to the fixed point subspace. Here again, the unitary transform $T$ and the associated block-diagonal structure is particularly useful and will provide significant computational speedups.


\section{Experimental Results} \label{sec:experimental_results}
In this section we report results obtained from a set of simulations that illustrate the efficacy of accounting for symmetry via the methods proposed in this paper.

In the first simulation experiment we generate a covariance matrix $\Sigma$ of size $p=50$ that is circulant and hence invariant with respect to the cyclic group. The first row of the circulant matrix is randomly generated (and then modified to ensure the matrix is positive definite). Figure \ref{fig:1} illustrates the rates of convergence that are observed in different matrix norms as the number of samples are increased. (For each sampling level, $20$ trials are conducted and the average error is reported). It is apparent that the estimate formed by projection of $\Sigma^n$ onto the fixed-point subspace far outperforms the estimate formed by $\Sigma^n$.
\begin{figure}  
\begin{center}  
\hspace{-.55in}
\includegraphics[scale=0.5]{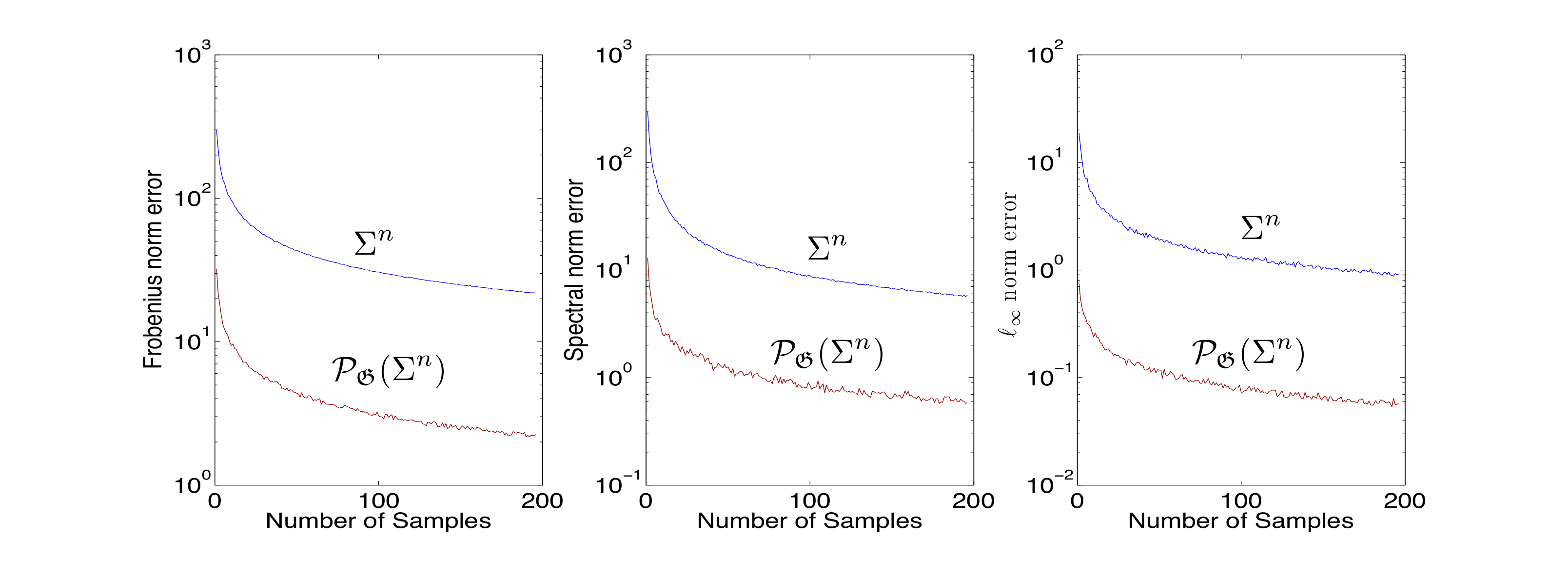}
\caption{\small A comparison of the rates of convergence for a covariance matrix that is invariant with respect to the cyclic group. The $\mf{G}$-empirical covariance matrix is a better estimate as compared to the empirical covariance matrix. }  
\label{fig:1}
\end{center}  
\end{figure}  

In the next simulation, we generate a \emph{sparse} circulant covariance matrix $\Sigma$  of size $p=50$ (so that it is again invariant with respect to the cyclic group). The entries as well as the sparsity pattern of the first row of the circulant matrix are random. The sparsity pattern is such that twenty percent of the matrix has non-zero entries. Figure \ref{fig:2} compares the performance obtained by different means of regularization. The first is symmetry-oblivious and uses the methodology proposed by Bickel and Levina \cite{BickelLevina2}, that thresholds the matrix at an appropriate level (that is proportional to $\sqrt{\frac{\log p}{n}}$). The second methodology is symmetry aware, and projects the empirical covariance matrix to the fixed-point subspace, followed by thresholding (proportional to $\sqrt{\frac{\log p}{pn}}$ as proposed in Section~\ref{sec:cov}). We compare the empirically observed error in different norms.
\begin{figure}  
\begin{center}  
\hspace{-.55in}
\includegraphics[scale=0.5]{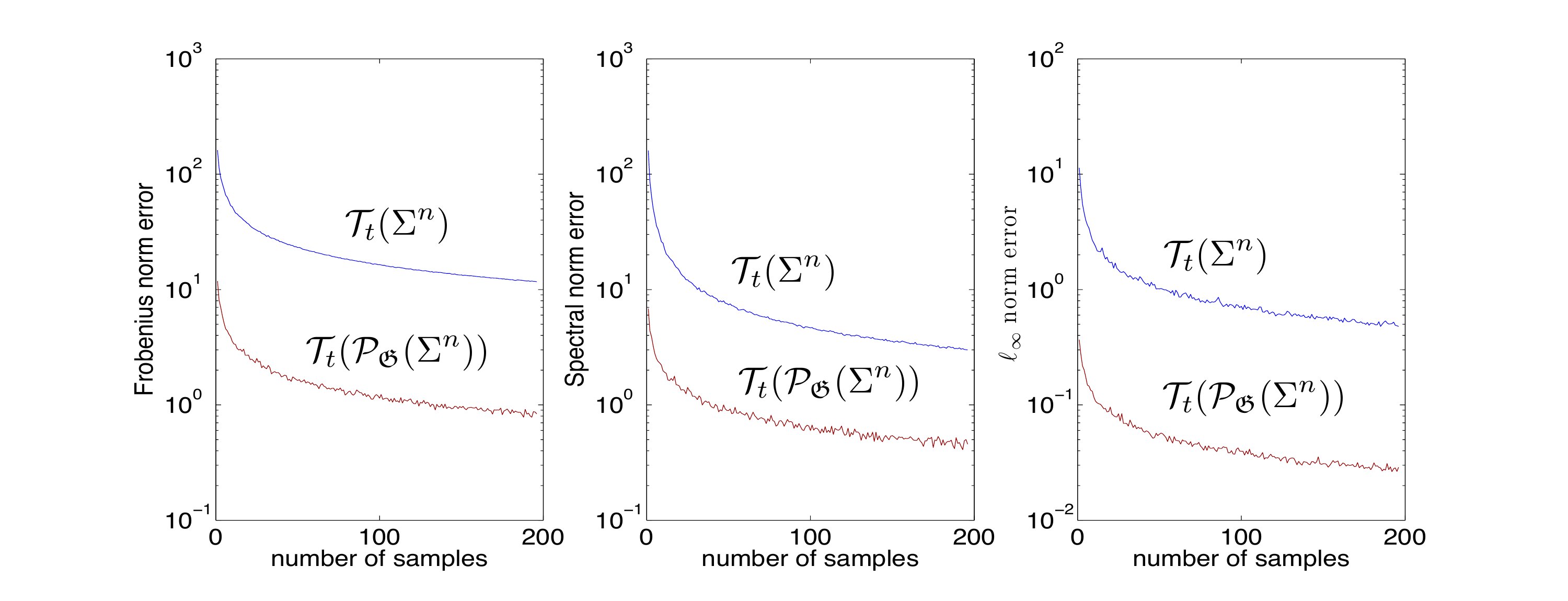}
\caption{\small A comparison of the rates of convergence for a sparse covariance matrix that is invariant with respect to the cyclic group. Projection followed by thesholding provides a better estimate. }  
\label{fig:2}
\end{center}  
\end{figure}  

In the last experiment, we consider the problem of recovering a sparse inverse covariance matrx (i.e. a graphical model) of size $p=50$. The graph corresponding to the graphical model is a cycle. The entries $\Theta^*_{ii}=1$ for all $i$. The edge weights of the cycle are all set to $\Theta^*_{ij}=-0.25$, the remaining entries of $\Theta^{*}$ are zero. We perform an experiment pertaining to \emph{model selection consistency}, i.e. recovering the structure of the graph from a small number of measurements. We compare two methodologies. The first is symmetry-oblivious, and uses the $\ell_1$ regularized $\log \det$ formulation proposed in \cite{RWRY} using only the empirical covariance matrix $\Sigma^n$ in the objective function. The second uses the formulation \eqref{eq:opt2} proposed in Section \ref{sec:inv_cov}, i.e. we first project $\Sigma^n$ onto the fixed-point subspace corresponding to the cyclic group, and then solve the $\ell_1$ regularized $\log \det$ problem. The empirically observed probabilities of correct model selection using these two methods are shown in Fig. \ref{fig:3}. This plot is generated by varing the number of samples from $n=5$ to $n=1505$ at increments of $100$ samples. For each sample level, we perfom $50$ experiments and report the average probability of success. The regularization parameter for the symmetry oblivious method is chosen proportional to $\sqrt{\frac{\log p}{n}}$, as proposed in \cite{RWRY}, and for the symmetry aware method it is chosen to be proportional to  $\sqrt{\frac{\log p}{np}}$, as proposed in Section \ref{sec:inv_cov}. Both the symmetry oblivious, as well as the symmetry aware formulations required the solution of $\ell_1$ regularized $\log \det$ problems. These were solved using the solver COVSEL \cite{daspremont}.

As seen in these experiments, for models containting symmetry the symmetry aware methodology proposed in this paper significantly outperforms symmetry oblivious methods.
\begin{figure}  
\begin{center}  
\includegraphics[scale=0.5]{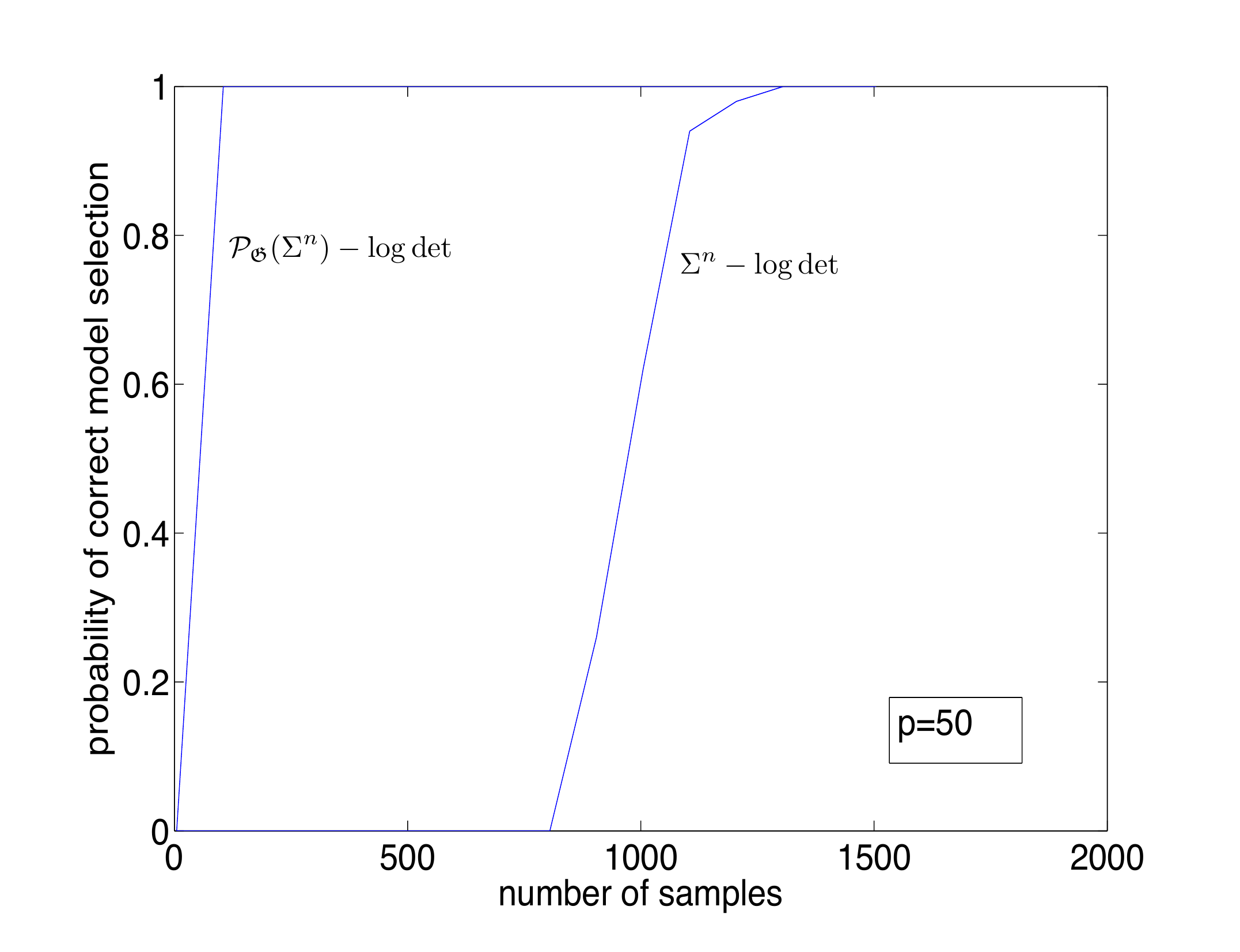}
\caption{\small A comparison of the model selection consistency obtained by $\mc{P}_{\mf{G}}(\Sigma^n)$ based $\log \det$ formulation versus a $\Sigma^n$ based $\log \det$ formulation. }  
\label{fig:3}
\end{center}  
\end{figure}  

\section{Conclusion} \label{sec:conclusion}
In this paper we have examined Gaussian models with symmetries in the high-dimensional regime. We have suggested a natural means of estimating a $\mf{G}$-invariant model from a small number of samples. The key idea is the formation of the projection of the empirical covariance matrix onto the fixed point subspace, which enjoys good statistical properties. As one of the main results of this paper, we have made precise the statistical gains enjoyed by this estimate by establishing explicit sample complexity bounds for the Frobenius, spectral, and $\ell_\infty$ norms. These sample complexity bounds depend naturally on certain group theoretic parameters that characterize the extent to which the group theoretic structure allows sample reuse. We have further used our sample complexity bounds to analyze sparse covariance models, and sparse graphical models that enjoy symmetry. Finally we provide experimental evidence that verifies our results via computer simulations.

\section*{Acknowledgement}
The authors would like to thank Pablo Parrilo, Benjamin Recht, Alan Willsky, and  Stephen Wright for helpful discussions.

\bibliographystyle{plain}
\bibliography{symm_graphical_models}

\appendix
\section{$\ell_\infty$ rates}
Before we present the formal proof, we remind the reader of some basic machinery related to pre-Gaussian random variables.
\begin{definition} \label{def:pregaussian}
A random variable $Y \in \R$ is called \emph{pre-Gaussian} \cite{Buldygin_Kozachenko} with parameters $\tau, \Lambda$ if
$$
\Ex \exp(\lambda Y) \leq \exp \l\frac{1}{2}\tau^2 \lambda^2 \r,
$$
for all $\lambda \in (-\Lambda, \Lambda)$.
\end{definition}
A useful fact about the pre-Gaussian variables (see \cite[Chapter 1]{Buldygin_Kozachenko}) is that if $Y$ is pre-Gaussian with parameters $\tau, \Lambda$,  the random variable $cY$ is pre-Gaussian with parameters $|c|\tau, \frac{\Lambda}{|c|}$. The following lemma characterizes the tail behavior of pre-Gaussian random variables.

\begin{lemma}\label{lemma:pregaussian_prop2}
Suppose $Y$ is a pre-Gaussian random variable with parameters $\tau, \Lambda$. Then
\begin{equation*}
\Prob  (|Y|> t) \leq \left\{ \begin{array}{ll}
2\exp\l -\frac{t^2}{2\tau^2} \r \qquad & \text{if } 0 < t < \Lambda \tau^2 \\
2\exp\left( -\frac{\Lambda t}{2} \right) \qquad & \text{if }  \Lambda \tau^2 < t.
\end{array} \right.
\end{equation*}
\end{lemma}
\begin{proof}
See \cite[Chapter 1]{Buldygin_Kozachenko}.
\end{proof}
The canonical examples of pre-Gaussian random variables are $\chi^2$ distributed random variables.
Indeed, let $D \in \R^{p \times p}$ be a diagonal positive definite matrix and let $X \sim \mc{N}(0,D)$ be a $\R^p$-valued Gaussian random variable. Then the random variable $X^T X-\Tr(D)$ is pre-Gaussian (see \cite[Lemma 5.1, Example 3.1]{Buldygin_Kozachenko}) with parameters
\begin{equation} \label{eq:indep_chi_square}
\tau^2=\Tr(D) \qquad  \Lambda=\frac{1}{2\max_i D_{ii}}.
\end{equation} 


\begin{lemma} \label{lemma:nonstandard_chi_square2}
Let $Y \in \R^{n} \sim \mc{N}(0,\Sigma)$. Let $Q=Q^T \succ 0$ and let $\lambda_{\max}(\cdot)$ denote the maximum eigenvalue. Then \small
$$\Ex \exp\left(\lambda(Y^{T}Q Y-\Tr(Q\Sigma)) \right)  \leq \exp \left(\frac{1}{2}\lambda^2\Tr(Q\Sigma) \right) \; \forall \lambda \in \left[-\Lambda, \Lambda \right],$$ \normalsize
where $\Lambda=\frac{1}{2  \lambda_{\max} \l\l\Sigma^{\frac{1}{2}}\r^TQ \Sigma^{\frac{1}{2}} \r}$.
\end{lemma}
\begin{proof}
For $Y\sim \mc{N}(0,\Sigma)$ let us define $X=Q^{\frac{1}{2}}Y$. Then $X \sim \mc{N}(0,Q^{\frac{1}{2}} \Sigma Q^{\frac{1}{2}})$ and $Y^TQY-\Tr(Q \Sigma) = X^TX-\Tr(Q \Sigma)$. Consider $R = U^T X$ where $U$ is unitary and has the property that $U^T Q^{\frac{1}{2}} \Sigma Q^{\frac{1}{2}} U=D$ and $D$ is diagonal. Applying \eqref{eq:indep_chi_square} to $R^T R -\Tr(Q \Sigma)=Y^TQY-\Tr(Q\Sigma)$, and using the fact that $\Tr (Q^{\frac{1}{2}} \Sigma Q^{\frac{1}{2}})=\Tr (Q \Sigma)$, and $\lambda_{\max} (Q^{\frac{1}{2}} \Sigma Q^{\frac{1}{2}})= \lambda_{\max} (\Sigma^{\frac{1}{2}} Q \Sigma^{\frac{1}{2}})$, we have the desired result.
\end{proof}

In the subsequent analysis, it will be more convenient to deal with normalized random variables, which we define as:
$$
\bar{X}^{(k)}_{g(i)}:=\frac{X^{(k)}_{g(i)}}{\sqrt{\Sigma_{ii}}}.
$$
Note that $\Ex \l\bar{X}^{(k)}_{g(i)} \r^2=1$. We also define:
$$
\rho_{ij}:=\frac{\Sigma_{ij}}{\sqrt{\Sigma_{ii}\Sigma_{jj}}},
$$
and note that by symmetry of the covariance matrix $\Sigma_{ii}=\Sigma_{g(i),g(i)}$ and hence $\rho_{ij}=\rho_{g(i),g(i)}$.
Define
\begin{equation}\label{eq:Yij}
Y_{ij}^{(k)}
=\sum_{(g(i),g(j)) \in \orb(i,j)} \bar{X}_{g(i)}^{(k)}\bar{X}_{g(j)}^{(k)}.
\end{equation}
We observe that 
\begin{equation} \label{eq:Uijk}
\begin{split}
Y_{ij}^{(k)}-|\orb(i,j)|\rho_{ij}&=\frac{1}{4} \left( U_{ij}^{(k)} -V_{ij}^{(k)} \right) \qquad \text{where,} \\
U_{ij}^{(k)}&=\sum_{(g(i),g(j)) \in \orb(i,j)} (\bar{X}_{g(i)}^{(k)}+\bar{X}_{g(j)}^{(k)})^2 -2|\orb(i,j)|(1+\rho_{ij}) \\
V_{ij}^{(k)}&=\sum_{(g(i),g(j)) \in \orb(i,j)} (\bar{X}_{g(i)}^{(k)}-\bar{X}_{g(j)}^{(k)})^2 -2|\orb(i,j)|(1-\rho_{ij}).\\
\end{split}
\end{equation}

For the covariance matrix of the normalized random variables occuring in the orbit of $(i,j)$ define  $\sigma_{ij}=\| \Sigma_{\mc{V}(i,j)} \| $, the spectral norm restricted to the variables occuring in the orbit. Note that $\sigma_{ij} \leq \|\Sigma \|$ and is hence bounded by a constant.
Also recall that $d_{ij}$ is the degree of the orbit $\orb(i,j)$.

\begin{lemma} \label{lemma:orbit_rate}
The following inequalities hold:
\begin{align*}
\Ex \exp(\lambda U_{ij}^{(k)}) &\leq \exp \l |\mc{O}(i,j)| \lambda^2(1+\rho_{ij})  \r\text{for all} \; \lambda \in \left[-\frac{1}{4d_{ij}\sigma_{ij}}, \frac{1}{4d_{ij}\sigma_{ij}} \right] \\
\Ex \exp(\lambda V_{ij}^{(k)}) &\leq \exp \l |\mc{O}(i,j)| \lambda^2(1-\rho_{ij})   \r\text{for all} \; \lambda \in \left[-\frac{1}{4d_{ij}\sigma_{ij}}, \frac{1}{4d_{ij}\sigma_{ij}} \right]. \\
\end{align*}
\end{lemma}
%
\begin{proof}
Fix a particular sample $k$, a pair $(i,j)$, and consider the corresponding $\orb(i,j)$ and $\mc{V}(i,j)$. Let $Z \in \R^{|\mc{V}(i,j)|}$ be the vector of all the distinct nodes that occur on the edge orbit $\orb(i,j)$. Let $W^+$ be the vector of all the terms $X_{g(i)}^{(k)}+X_{g(j)}^{(k)}$, and $W^-$ be the vector of all the terms $X_{g(i)}^{(k)}-X_{g(j)}^{(k)}$. Then there are matrices $A^+$ and $A^-$ (node-edge incidence matrices) such that 
$$
W^+=A^+Z \qquad W^-=A^-Z.
$$
Note that $U_{ij}^{(k)}=(W^{+})^TW^+ - 2|\orb(i,j)|(1+\rho_{ij})$ and $V_{ij}^{(k)}=(W^-)^TW^- - 2|\orb(i,j)|(1-\rho_{ij})$.
Let $\Sigma_Z$ be the submatrix of the covariance matrix associated to $Z$. Then we note that we can write $$(W^{+})^TW^+=Z^T(A^+)^{T}A^+Z$$
where $Z \in \R^{|\mc{V}(i,j)|} \sim \mc{N}(0,\Sigma_Z)$. Note that by definition of $A^+$ and $Z$, $\Tr \l \l A^+ \r^T A^+ \Sigma_Z \r=2\orb(i,j)(1+\rho_{ij})$ (similar equality holds for $A^-$). 
To be able to invoke Lemma \ref{lemma:nonstandard_chi_square2} we need to bound two quantities: 
\begin{align*}
 \lambda_{\max} \l (\Sigma_Z^{\frac{1}{2}})^{T}(A^+)^{T}A^+\Sigma_Z^{\frac{1}{2}} \r  \text{ and }  \lambda_{\max} \l (\Sigma_Z^{\frac{1}{2}})^{T}(A^-)^{T}A^-\Sigma_Z^{\frac{1}{2}} \r. \end{align*}

Note that 
\begin{align*}
\lambda_{\max} \l (\Sigma_Z^{\frac{1}{2}})^{T}(A^+)^{T}A^+\Sigma_Z^{\frac{1}{2}} \r &\leq \| \l A^{+}\r^T A^+ \| \|\Sigma_Z \| \\
& \leq \sigma_{ij}  \| \l A^{+}\r^T A^+ \|.
\end{align*}
By the defninition of $A^+$, we have
\begin{align*}
\| \l A^{+}\r^T A^+ \| &= \max _{x \neq 0} \frac{\sum_{(g(i), g(j))\in \orb(i,j)}(x(i)+x(j))^2}{\sum_{k \in \mc{V}(i,j)} x_k^2} \\
& \leq \max _{x \neq 0} \frac{2\sum_{(g(i), g(j))\in \orb(i,j)} \l x(i)^2+x(j)^2\r}{\sum_{k \in \mc{V}(i,j)} x_k^2} \\
& \leq \max _{x \neq 0} \frac{2 d_{ij}\sum_{k \in \mc{V}(i,j)} x_k^{2} }{\sum_{k \in \mc{V}(i,j)} x_k^2} \\
&\leq 2d_{ij}.
\end{align*}
Hence 
$\lambda_{\max} \l (\Sigma_Z^{\frac{1}{2}})^{T}(A^+)^{T}A^+\Sigma_Z^{\frac{1}{2}} \r \leq 2d_{ij}\sigma_{ij}$.
(Similar inequalities hold for $A^-$).

Now using Lemma \ref{lemma:nonstandard_chi_square2}, the result follows.
\end{proof}
\begin{lemma}
$\frac{\hat{\Sigma}_{\mc{R},ij}^{(k)}}{\sqrt{\Sigma_{ii}\Sigma_{jj}}} = \frac{1}{|\orb(i,j)|} Y_{ij}^{(k)}.$
\end{lemma}
\begin{proof}
\begin{align*}
\hat{\Sigma}_{\mc{R},ij}&=  \frac{1}{|\mathfrak{G}|} \sum_{g \in \mathfrak{G}}  X_{g(i)}X_{g(j)}\\
&=\frac{1}{|\mathfrak{G}|} \sum_{(g(i), g(j)) \in \orb(i,j)} n_{ij} X_{g(i)}X_{g(j)},
\end{align*}
where $n_{ij}$ is the number of times edge $(i,j)$ (and hence all other edges in the orbit) are repeated in the Reynolds average. Now note that the size of the orbit divides the size of the group, and in fact $\frac{|\mathfrak{G}|}{|\orb(i,j)|}=n_{ij}$ (in fact both these are equal to the cardinality of the stabilizer of the edge $(i,j)$). Dividing throughout by $\sqrt{\Sigma_{ii}\Sigma_{jj}}$, we get the desired result.
\end{proof}

Let 
\begin{align} \label{eq:Uij}
U_{ij} =\frac{1}{4n|\orb(i,j)|} \sum_{k=1}^{n} U_{ij}^{(k)} \qquad
V_{ij} =\frac{1}{4n|\orb(i,j)|} \sum_{k=1}^{n} V_{ij}^{(k)} \qquad
Y_{ij} =\frac{1}{4n|\orb(i,j)|} \sum_{k=1}^{n} Y_{ij}^{(k)}.
\end{align}

\begin{lemma}\label{lemma:prob_union_bnd}
$$
\Prob \l |\hat{\Sigma}_{ij}-\Sigma_{ij}|>\sqrt{\Sigma_{ii}\Sigma_{jj}}t \r \leq \Prob \l |U_{ij}|>\frac{t}{2} \r + \Prob \l |V_{ij}|>\frac{t}{2} \r.
$$
\end{lemma}
\begin{proof}
We have 
\begin{align*}
\frac{\hat{\Sigma}_{ij}-\Sigma_{ij}}{\sqrt{\Sigma_{ii}\Sigma_{jj}}}&=\frac{1}{n|\orb(i,j)|}\sum_{k=1}^{n}Y_{ij}^{(k)} - \rho_{ij} \\
&=\frac{1}{n}\sum_{k=1}^{n} \frac{1}{|\orb(i,j)|} \l Y_{ij}^{(k)} - |\orb(i,j)| \rho_{ij} \r \\
&=\frac{1}{n}\sum_{k=1}^{n} \frac{1}{4|\orb(i,j)|}  \l U_{ij}^{(k)}-V_{ij}^{(k)}  \r  \qquad \text{(using \eqref{eq:Uijk})} \\
&=U_{ij}-V_{ij} \qquad \text{(by \eqref{eq:Uij}).}
\end{align*}
Now $$\left\{|\hat{\Sigma}_{ij}-\Sigma_{ij}|>\sqrt{\Sigma_{ii}\Sigma_{jj}}t \right\} = \left\{ |U_{ij}-V_{ij}|>t \right\} \subseteq \left\{|U_{ij}|>\frac{t}{2} \right\} \cup \left\{|V_{ij}|>\frac{t}{2} \right\},$$
from which the conclusion follows.
\end{proof}

\begin{proof}[Proof of Theorem \ref{theorem:rate}]
Note that 
$\Prob \l |U_{ij}| > \frac{t}{2} \r \leq 2  \Prob \l U_{ij} > \frac{t}{2} \r$. 
Now $U_{ij}=\frac{1}{4n|\orb(i,j)|} \sum_{k=1}^{n}U_{ij}^{(k)}$, so that
$$
\Ex \exp (\lambda U_{ij})=\Ex \exp \l \frac{1}{4n|\orb(i,j)|} \sum_{k=1}^{n}U_{ij}^{(k)}\r.
$$
From Lemma \ref{lemma:orbit_rate}, $$\Ex \exp(\lambda U_{ij}^{(k)}) \leq \exp \l |\mc{O}(i,j)| \lambda^2(1+\rho_{ij})  \r\text{for all} \; \lambda \in \left[-\frac{1}{4d_{ij}\sigma_{ij}}, \frac{1}{4d_{ij}\sigma_{ij}} \right].$$
By rescaling, we have
$$\Ex \exp \l \frac{1}{4n|\orb(i,j)|}\lambda U_{ij}^{(k)} \r \leq \exp \l \frac{(1+\rho_{ij})}{16 |\orb(i,j)|n^2} \lambda^2  \r\text{for all} \; \lambda \in \left[-\frac{n|\orb(i,j)|}{d_{ij}\sigma_{ij}}, \frac{n|\orb(i,j)|}{d_{ij}\sigma_{ij}} \right].$$
Since the random variables $U_{ij}^{(k)}$ are independent and identically distributed (they correspond to different samples) for $k=1, \ldots, n$, we have that 
$$\Ex \exp \l \lambda U_{ij} \r \leq \exp \l \frac{(1+\rho_{ij})}{16 |\orb(i,j)|n} \lambda^2  \r\text{for all} \; \lambda \in \left[-\frac{n|\orb(i,j)|}{d_{ij}\sigma_{ij}}, \frac{n|\orb(i,j)|}{d_{ij}\sigma_{ij}} \right].$$
Thus $U_{ij}$ is pre-Gaussian with parameters $\tau^2=\frac{(1+\rho_{ij})}{16 |\orb(i,j)|n}$ and $\Lambda= \frac{n|\orb(i,j)|}{d_{ij}\sigma_{ij}}$. From Lemma \ref{lemma:pregaussian_prop2} we have
\begin{equation*}
\Prob \l |U_{ij}| > \frac{t}{2} \r \leq 
\left\{  \begin{array}{l}
2\exp \l-\frac{2t^2 n |\orb(i,j)|}{1+\rho_{ij}}  \r, \; \text{for } 0 \leq t \leq \frac{1+\rho_{ij}}{16d_{ij}\sigma_{ij}} \\
2 \exp \l-\frac{n |\orb(i,j)|t}{2d_{ij}\sigma_{ij}} \r,  \; \text{for } t > \frac{1+\rho_{ij}}{16d_{ij}\sigma_{ij}}.
\end{array}\right.
\end{equation*}
Another way of rewriting the above is:
\begin{equation*}
\Prob \l |U_{ij}| > \frac{t}{2} \r \leq 
\max  \left\{
2\exp \l-\frac{2t^2 n |\orb(i,j)|}{1+\rho_{ij}}  \r, 
2 \exp \l-\frac{n |\orb(i,j)|t}{2d_{ij}\sigma_{ij}} \r
\right\}.
\end{equation*}

By a similar argument, 
\begin{equation*}
\Prob \l |V_{ij}| > \frac{t}{2} \r \leq 
\max  \left\{
2\exp \l-\frac{2t^2 n |\orb(i,j)|}{1-\rho_{ij}}  \r, 
2 \exp \l-\frac{n |\orb(i,j)|t}{2d_{ij}\sigma_{ij}} \r
\right\}.
\end{equation*}
Let $C_1=\min_{i,j} \left\{\frac{2}{1+\rho_{ij}},\frac{2}{1-\rho_{ij}} \right\}$, and $C_2=\min_{i,j} \left\{\frac{2}{\sigma_{ij}},\frac{2}{\sigma_{ij}} \right\}$. (Note that all these quantities are constants since the spectral norm $\|\Sigma \|$ is bounded by a constant by assumption). Using Lemma \ref{lemma:prob_union_bnd}, we get the following.
$$
\Prob \l |\hat{\Sigma}_{ij}-\Sigma_{ij}|>\sqrt{\Sigma_{ii}\Sigma_{jj}}t \r \leq 
\max  \left\{
\exp \l-C_1 n |\orb(i,j)|t^2  \r, 
 \exp \l-\frac{C_2n |\orb(i,j)|t}{d_{ij}}\r
\right\}.
$$
Noting that $\sqrt{\Sigma_{ii}\Sigma_{jj}}$ is bounded by a constant, we replace $t$ by $\frac{t}{\sqrt{\Sigma_{ii}\Sigma_{jj}}}$, and alter the constants in the exponent suitably to obtain the required result.
\end{proof}

\section{Proof of Theorem \ref{theorem:covariance}}
We define the following function that captures the $\ell_\infty$ norm sample complexity
\begin{equation}
\delta(n,p):=  \max \left\{ \sqrt{\frac{\log p}{n\orb}}, \frac{\log p}{n\orb_d}\right\}.
\end{equation}
It will be convenient to work with the notation $\hat{\Sigma}=\mc{P}_{\mf{G}} \l \Sigma^n \r$.
\begin{proof}
The proof is almost identical to the proof of \cite[Theorem 1]{BickelLevina2}. We present it below for the sake of completeness.
Note that from Corollary \ref{cor:linf_rate1}, 
$$
\max_{i,j} |\hat{\Sigma}_{ij}-\Sigma_{ij} |=O_{P} \l \delta(n,p) \r.
$$
Now we note that 
$$
\|\mc{T}_t(\hat{\Sigma})-\Sigma \| \leq \|\mc{T}_{t}(\Sigma)-\Sigma \| + \|\mc{T}_t(\hat{\Sigma})-\mc{T}_t(\Sigma) \|,
$$
and we will focus on bounding the different terms. 

Recall that for a symmetric matrix $M$, $\|M\| \leq \max_{j= 1, \ldots ,p} \sum_{i=1}^p | M_{ij}|$. (This is because $\|M\| \leq \l \| M \|_{1,1} \|M \|_{\infty, \infty} \r^{\frac{1}{2}}=\|M \|_{1,1}$). As a consequence, the first term is bounded by 
$$
\max_{i} \sum_{j=1}^{p} |\Sigma_{ij}|\one(|\Sigma_{ij}| \leq t) \leq t^{1-q}c_0(p).
$$
The second term can be bounded as
\begin{align*}
\|\mc{T}_{t}(\hat{\Sigma})-\mc{T}_t(\Sigma) \| & \leq \max_{i} \sum_{j=1}^{p} |\hat{\Sigma}_{ij}| \one(|\hat{\Sigma}_{ij}| \geq t, |\Sigma_{ij}|<t)+ \max_{i} \sum_{j=1}^{p} |\Sigma_{ij}| \one(|\hat{\Sigma}_{ij}| < t, |\Sigma_{ij}| \geq t) \\ &+\max_i \sum_{j=1}^{p}|\hat{\Sigma}_{ij}-\Sigma_{ij}| \one(|\hat{\Sigma}_{ij}| \geq t, |\Sigma_{ij}| \geq t) \\
&=\text{I}+\text{II}+\text{III}.
\end{align*}
Now note that 
$$
\text{III} \leq \max_{i,j} |\hat{\Sigma}_{ij}-\Sigma_{ij}| \max_{i} \sum_{j=1}^{p}|\hat{\Sigma}_{ij}|^qt^{-q} =O_{P} \l c_0(p) t^{-q} \delta(n,p)\r. 
$$
For term II we note that
\begin{align*}
\text{II} & \leq \max_{i} \sum_{j=1}^{p} (|\hat{\Sigma}_{ij} -\Sigma_{ij}| +| \hat{\Sigma}_{ij}|) \one(|\hat{\Sigma}_{ij}| < t, |\Sigma_{ij}| \geq t)  \\
& \leq\max_{i,j} |\hat{\Sigma}_{ij} -\Sigma_{ij}| \sum_{j=1}^{p}\one(|\Sigma_{ij}| \geq t) +t \max_{i} \sum_{j=1}^{p} \one(|\Sigma_{ij}|\geq t) \\
&=O_{P} \l c_0(p)t^{-q} \delta(n,p) + c_0(p)t^{1-q}\r.
\end{align*}
To bound I, we pick $\gamma \in (0,1)$ and note that
\begin{align*}
\text{I} & \leq \max_{i} \sum_{j=1}^p |\hat{\Sigma}_{ij}-\Sigma_{ij}|\one(|\hat{\Sigma}_{ij}| \geq t, |\Sigma_{ij}| < t)  + \max_{i} \sum_{j=1}^p |\Sigma_{ij}| \one(|\Sigma_{ij}| < t)  \\
& \leq \max_{i} \sum_{j=1}^p |\hat{\Sigma}_{ij}-\Sigma_{ij}|\one(|\hat{\Sigma}_{ij}| \geq  t, |\Sigma_{ij}| \leq \gamma t) +
\max_{i} \sum_{j=1}^p |\hat{\Sigma}_{ij}-\Sigma_{ij}|\one(|\hat{\Sigma}_{ij}| \geq  t, \gamma t|\Sigma_{ij}| \leq t)  \\
& \qquad \qquad  \qquad + t^{1-q}c_0(p) \\
&  \leq  \max_{i,j}   |\hat{\Sigma}_{ij}-\Sigma_{ij}| \l \sum_{j=1}^p \one (|\hat{\Sigma_{ij}}-\Sigma_{ij} |>(1-\gamma)t) + c_{0}(p)(\gamma t)^{-q}  \r  + t^{1-q}c_0(p)
\end{align*}
\begin{align*}
&  \leq  \max_{i,j}   |\hat{\Sigma}_{ij}-\Sigma_{ij}|  c_{0}(p)(\gamma t)^{-q} + t^{1-q}c_0(p) \; \text{(with high probability)} \\
&=O_{P} \l c_{0}(p)(\gamma t)^{-q} \delta(n,p) +t^{1-q}c_0(p) \r.
\end{align*}
The second to last inequality follows from the fact that:
\begin{align*}
\Prob \l \max_{i} \sum_{j=1}^p \one (|\hat{\Sigma_{ij}}-\Sigma_{ij} |>(1-\gamma)t) >0  \r &=\Prob \l\max_{i,j} |\hat{\Sigma}_{ij}-\Sigma_{ij}|>(1-\gamma)t \r\\
&\leq \frac{1}{p^c} \;
\end{align*}
 for $t$ as chosen in the statement of the theorem. Combining all these inequalities we get the required result.
\end{proof}
 \end{document}